\documentclass[11pt,reqno]{amsart}
\usepackage{amsmath,amssymb,amsthm}
\usepackage{color}
\usepackage[tmargin=1in,bmargin=1in,rmargin=1in,lmargin=1in]{geometry}
\usepackage[breaklinks=true]{hyperref}
\usepackage{enumitem}

\theoremstyle{plain}
\newtheorem{theorem}{Theorem}[section]

\newtheorem{lemma}[theorem]{Lemma}
\newtheorem{proposition}[theorem]{Proposition}
\newtheorem{question}[theorem]{Question}

\theoremstyle{definition}
\newtheorem{definition}[theorem]{Definition}

\newtheorem{remark}[theorem]{Remark}

\numberwithin{equation}{section}
\numberwithin{table}{section}
\allowdisplaybreaks

\newcommand{\inc}[2]{{#1}^{#2, \uparrow}}
\newcommand{\bt}{\mathbf{t}}
\newcommand{\bx}{\mathbf{x}}
\newcommand{\by}{\mathbf{y}}
\renewcommand{\geq}{\geqslant}
\renewcommand{\leq}{\leqslant}
\DeclareMathOperator{\Id}{Id}
\DeclareMathOperator{\TN}{TN}
\DeclareMathOperator{\TP}{TP}
\DeclareMathOperator{\STN}{SymTN}
\DeclareMathOperator{\STP}{SymTP}

\begin{document}

\title{Multivariate transforms of total positivity}

\author{Sujit Sakharam Damase}
\address[S.S.~Damase]{Department of Mathematics, Indian Institute of
Science, Bangalore 560012, India; PhD student}
\email{\tt sujits@iisc.ac.in}

\author{Apoorva Khare}
\address[A.~Khare]{Department of Mathematics, Indian Institute of
Science, Bangalore, India and Analysis \& Probability Research Group,
Bangalore, India}
\email{\tt khare@iisc.ac.in}

\date{\today}



\begin{abstract}
Belton--Guillot--Khare--Putinar [\textit{J.\ d'Analyse Math.}\ 2023] classified the post-composition operators that preserve TP/TN kernels of each specified order. We explain how to extend this from preservers to transforms, and from one to several variables. Namely, given arbitrary nonempty totally ordered sets $X,Y$, we characterize the transforms that send each tuple of kernels on $X \times Y$ that are TP/TN of orders $k_1, \dots, k_p$, to a TP/TN kernel of order $l$, for arbitrary positive integers (or infinite) $k_j$ and $l$. An interesting feature is that to preserve TP (or TN) of order $2$, the preservers are products of individual power (or Heaviside) functions in each variable; but for all higher orders, the preservers are powers in a single variable. We also classify the multivariate transforms of symmetric TP/TN kernels; in this case it is the preservers of TP/TN of order 3 that are multivariate products of power functions, and of order 4 that are individual powers. The proofs use generalized Vandermonde kernels, Hankel kernels, (strictly totally positive) P\'olya frequency functions, and a kernel studied recently but tracing back to works of Schoenberg [\textit{Ann.\ of Math.}\ 1955] and Karlin [\textit{Trans.\ Amer.\ Math.\ Soc.}\ 1964].
\end{abstract}

\maketitle

\section{Introduction and main results}

Given totally ordered sets $X$ and $Y$, we say that a kernel $K : X \times Y \to (0,\infty)$ is \textit{totally positive} if for any choice of points $x_1 < \cdots < x_n$ in $X$ and $y_1 < \cdots < y_n$ in $Y$, the matrix $(K(x_i, y_j))_{i,j=1}^n$ has positive determinant. These matrices and kernels (and their non-strict ``totally nonnegative'' counterparts) have featured in many different areas of mathematics, including
in analysis (morally going back to Laguerre) and differential equations \cite{AESW,FP, Loewner55, IJS-PF},
classical works on particle systems \cite{GK37,GK50},
matrix analysis \cite{Schoenberg55,Whitney},
probability and statistics \cite{Efron, KarlinTAMS, KP},
a comprehensive monograph \cite{SK-TP} and more recent volumes \cite{FJ,TP,Pinkus},
approximation theory \cite{SW53},
Gabor analysis \cite{GRS},
integrable systems \cite{KW},
representation theory \cite{Lusztig, Rietsch},
cluster algebras \cite{BZ,FZ},
the totally non-negative Grassmannian \cite{GKL,Postnikov},
and combinatorics \cite{Br} -- to cite but a few influential works in numerous subfields.

A prominent problem in the literature has been that of understanding the preservers of various structures, including forms of positivity. For instance, perhaps the earliest results on preservers go back to Frobenius~\cite{Frob}, who characterized in 1897 the linear maps on matrix algebras which were determinant-preserving. Here we are interested in the theme of positivity preservers, which involves preserving the weaker notion of positivity of (determinants of) various submatrices of kernels. In this vein, in 1925 P\'olya and Szeg\H{o}~\cite{polya-szego} used the Schur product theorem \cite{Schur1911} to show that convergent power series with nonnegative coefficients, when applied entrywise to positive semidefinite matrices of all sizes -- or positive kernels on an infinite domain -- preserve positive semidefiniteness. The converse to this result -- that there exist no other such preservers -- was shown by Schoenberg~\cite{Schoenberg42} and Rudin \cite{Rudin59}, and later by a host of other mathematicians in various settings. Schoenberg's motivations in the above work came from metric geometry and from understanding positive definite functions on various homogeneous spaces. These are \textit{transforms} that send distance matrices to (conditionally) positive semidefinite matrices. We refer the reader to e.g.\ the twin surveys \cite{BGKP-survey1,BGKP-survey2} or the monograph \cite{AK-book} for more details.

The brief literature survey above describes functions of one variable which preserve or transform various structures related to positivity. There are also \textit{multivariate} analogues of the above classification results, of which we list a selection here.
\begin{itemize}
\item Norvidas \cite{N85} classified the post-composition preservers which take a tuple of continuous positive definite functions on certain LCA groups to a continuous positive definite function.
\item FitzGerald--Micchelli--Pinkus \cite{fitzgerald} classified the multivariate entrywise preservers and transforms taking tuples of (conditionally) positive matrices to (conditionally) positive matrices.
\item Belton--Guillot--Khare--Putinar \cite{BGKP-hankel} strengthened the preceding results for multivariate positivity preservers.
\item Recently, Belton et al extended this ``multivariate Schoenberg theorem'' to accommodate matrices with negative eigenvalues. Namely, they classified in \cite{BGKP-inertia} all multivariate entrywise maps $F : \mathbb{R}^p \to \mathbb{R}$ such that if $A_1, \dots, A_p$ are real symmetric $n \times n$ matrices (for any $n \geq 1$) with at most $k_1, \dots, k_p$ negative eigenvalues, then $F[ A_1, \dots, A_p]$ has at most $l$ negative eigenvalues -- for arbitrary (but fixed) integers $k_1, \dots, k_p, l \geq 0$.
\end{itemize}

The present work is at the intersection of these two evergreen themes: total positivity and (multivariate) preservers. More strongly, we seek to characterize the post-composition transforms of total positivity and of total nonnegativity -- of varying orders $k_1, \dots, k_p, l$ as above. In this regard, the one-variable case -- more precisely, classifying the preservers -- was resolved by Belton--Guillot--Khare--Putinar in  2023 \cite{BGKP-TN} (answering a ``univariate'' question by P.~Deift to A.K.). The present work provides the ``multivariate'' answers, and is thus a sequel to \cite{BGKP-TN}; additionally, it extends the results in \cite{BGKP-TN} from preservers to transforms. We begin by setting notation, and then describe our main results.

\begin{definition}
Fix nonempty totally ordered sets $X, Y$ and an integer $k \geq 1$.
\begin{enumerate}%[leftmargin = *]
\item Denote by $\inc{X}{k}$ all $k$-tuples of strictly increasing
elements from $X$:
\[
\inc{X}{k} := %
\{ ( x_1, \ldots, x_k ) \in X^k : x_1 < \cdots < x_k \}.
\]

\item A kernel $K : X \times Y \to [0,\infty)$ is said to be
\textit{totally nonnegative of order $k$}, denoted by $K \in \TN^{(k)}_{X
\times Y}$, if for all integers $1 \leq r \leq k$ and tuples $\bx \in
\inc{X}{r}, \by \in \inc{Y}{r}$, the ``submatrix''
\[
K[ \bx; \by ] := ( K(x_i, y_j) )_{i,j=1}^r
\]
has nonnegative determinant. We say $K$ is \textit{totally nonnegative},
denoted by $K\in \TN_{X \times Y}$ or $\TN^{(\infty)}_{X \times Y}$ if the above
holds for all integers $k \geq 1$.

\item Similarly, if the determinants are all positive then we say
$K$ is \textit{totally positive of order $k$}, denoted by $K \in
\TP^{(k)}_{X \times Y}$ -- or \textit{totally positive} if this holds for
all $k$, denoted by $K \in \TP_{X \times Y}$ or $\TP^{(\infty)}_{X \times
Y}$.

\item If $X,Y$ are finite sets, of sizes $m,n$ respectively, we will also
use $\TN_{m \times n}^{(k)}$ or $\TP_{m \times n}^{(k)}$.
\end{enumerate}
\end{definition}
Throughout this work, we adopt the convention that $0^{\alpha} = 0$ if $\alpha>0$, while $0^{0} := 1$. Also, $[p] :=\{1,\dots, p\}$ for all integers $ p\in \mathbb{Z}_{>0}$.

In recent work \cite{BGKP-TN}, Belton--Guillot--Khare--Putinar classified
the post-composition maps $F[-]$ that preserve $\TN^{(k)}$ kernels on $X
\times Y$. In other words, if $K \in \TN^{(k)}_{X \times Y}$ then $F
\circ K \in \TN^{(k)}_{X \times Y}$.

The goal of this work is to explore this problem with greater flexibility
-- wherein the orders of total positivity $(k,k)$ are now allowed to
differ so that one now seeks transforms, not just preservers -- and more strongly, to tackle this classification problem in several variables. To state this requires the following notation.

\begin{definition}
Given $X,Y$ as above and an integer $p \geq 1$, a function $F(\bt)$ of
$\bt = (t_1, \dots, t_p) \in [0,\infty)^p$ acts on a tuple of totally
nonnegative (of some orders) kernels ${\bf K} = (K_1, \dots, K_p) : X
\times Y \to [0,\infty)^p$ via:
\[
F[-] : {\bf K} \mapsto F \circ {\bf K};
\qquad
(x,y) \mapsto F(K_1(x,y), \dots, K_p(x,y)).
\]
The analogous notion makes sense where the kernels are $\TP$ and $F$ has
domain $(0,\infty)^p$.
\end{definition}

Of course, if the domains $X,Y$ are finite then the kernels are matrices, and the post-composition transforms are precisely entrywise functions applied to these (tuples of) matrices.

We are interested in understanding when $F \circ {\bf K}$ is $\TN$ or
$\TP$ (of some order) if $K_1, \dots, K_p$ are.
As we see below, this more challenging question requires additional
tools from measure theory and analysis, as well as novel arguments involving (regular) P\'olya frequency functions -- and in this work, we
provide a complete solution to the following

\begin{question}\label{Qmain}
Given nonempty totally ordered sets $X,Y$, an integer $p \geq 1$, and
arbitrary integers $1 \leq k_1, \dots, k_p, l \leq \infty$, classify all
entrywise transforms $F: [0,\infty)^{p} \to [0,\infty)$ satisfying
\begin{equation}\label{ETNq}
F[-] : \TN^{(k_1)}_{X \times Y} \times \cdots \times \TN^{(k_p)}_{X
\times Y} \to \TN^{(l)}_{X \times Y},
\end{equation}
and similarly, all entrywise transforms $F: (0,\infty)^{p} \to (0,\infty)$ satisfying
\begin{equation}\label{ETPq}
F[-] : \TP^{(k_1)}_{X \times Y} \times \cdots \times \TP^{(k_p)}_{X
\times Y} \to \TP^{(l)}_{X \times Y}.
\end{equation}
\end{question}

We first state the complete classification in the $\TN$ case, and then move on to the $\TP$ case. In the case of one variable studied in \cite{BGKP-TN}, the univariate preservers of $\TN^{(2)}$ matrices were -- up to rescaling -- the powers $t^\alpha$ for $\alpha \geq 0$ (with $0^0 = 1$), and the Heaviside indicator ${\bf 1}_{t>0}$. The first nontrivial part of the next result says that in the multivariate case, the corresponding transforms are precisely products of such functions in each variable.

\begin{theorem}[$\TN$ case]\label{TmainTN}
Fix nonempty totally ordered sets $X, Y$ and an integer $p \geq 1$. Now
fix integers $1 \leq k_1, \dots, k_p, l \leq \infty$, and define
$\mathcal{N} := \min \{ |X|, |Y|, l \} \in \mathbb{Z}_{>0} \sqcup \{\infty\}$.
\begin{enumerate}
\item If $\mathcal{N} = 1$, then the maps $F$
satisfying~\eqref{ETNq} are all functions $: [0,\infty)^p \to
[0,\infty)$.

\item If $\mathcal{N} = 2$, then $F : [0,\infty)^p \to [0,\infty)$ satisfies \eqref{ETNq} if and only if there exists a subset of indices $J \subseteq [p]$ and nonnegative constants $c$ and $\{ \alpha_j : j \in J \}$ such that
\[
F({\bf t}) = c \prod_{j \in J} t_j^{\alpha_j} \prod_{i \not\in J} {\bf 1}_{t_i > 0}.
\]
Moreover, if $k_j < \mathcal{N}$ for any $j \in [p]$, then $j \in J$ and $\alpha_j = 0$.

\item If $\mathcal{N} = 3$, then the functions $F$
satisfying~\eqref{ETNq} are the nonnegative constants and the individual
powers $F(\bt) = c t_j^{\alpha_j}$ for a single $j \in [p]$, with $c >
0, \alpha_j \geq 1$, and $k_j \geq \mathcal{N}$.

\item In all other cases -- i.e.\ if $4 \leq \mathcal{N} \leq \infty$,
the functions $F$ satisfying~\eqref{ETNq} are the nonnegative constants
and the homotheties $F(\bt) = c t_j$ for a single $j \in
[p]$, with $c > 0$ and $k_j \geq \mathcal{N}$.
\end{enumerate}
\end{theorem}

\begin{remark}
The point of defining $\mathcal{N}$ is that one works only with minors of $F \circ {\bf K}$ that are of size $\mathcal{N} \times \mathcal{N}$ or smaller.
\end{remark}

Despite the fact that $X,Y$ are arbitrary sets, the above classification of $\TN$ kernel transforms uses mostly matrix arguments. The key is to use smaller matrices (or kernels on smaller domains) to build up to bigger ones, because one can ``pad by zeros’’ any $\TN$ kernel on a smaller domain, to arrive at a $\TN$ kernel on $X \times Y$.

In contrast, the classification of $\TP$ kernel transforms turns out to be more challenging. As one can no longer pad kernels by zeros, one cannot directly go from $\TP$ kernels on smaller domains to ``bigger’’ ones. Furthermore, one needs to work with $\TP$ kernels on possibly infinite sets $X,Y$. Thus we employ (strictly totally positive) P\'{o}lya frequency functions, generalized Vandermonde kernels, and a kernel which can be traced back to works of Schoenberg \cite{Schoenberg55} and Karlin \cite{KarlinTAMS} as well as recent papers by Jain \cite{Jain, TJ} and Khare \cite{AK}. In this work, we show:

\begin{theorem}[$\TP$ case]\label{TmainTP}
Fix nonempty totally ordered sets $X,Y$ and an integer $p \geq 1$. Now
fix integers $1 \leq k_1, \dots, k_p, l \leq \infty$, and define
$\mathcal{N} := \min \{ |X|, |Y|, l \} \in \mathbb{Z}_{>0}\sqcup \{\infty\}$.
\begin{enumerate}
\item If $\mathcal{N} = 1$, then the maps $F$
satisfying~\eqref{ETPq} are all functions $: (0,\infty)^p \to
(0,\infty)$.

\item If $\mathcal{N} = 2$, then the maps $F : (0,\infty)^p \to
(0,\infty)$ satisfying~\eqref{ETPq} are the ``mixed power functions''
$F(\bt) = c \bt^{\boldsymbol{\alpha}} = c \prod_{j=1}^p t_j^{\alpha_j}$
for some reals $c > 0 $ and $\alpha_j \geq 0$,
with $\boldsymbol{\alpha}\neq {\bf 0}$, and $\alpha_j = 0$ if $k_j < \mathcal{N}$.

\item If $\mathcal{N} = 3$, then the functions $F$
satisfying~\eqref{ETPq} are the individual
powers $F(\bt) = c t_j^{\alpha_j}$ for a single $j \in [p]$, with $c >
0, \alpha_j \geq 1$, and $k_j \geq \mathcal{N}$.

\item In all other cases -- i.e.\ if $4 \leq \mathcal{N} \leq \infty$,
we further assume that if $X,Y$ are both infinite but $l < \infty$ then all $k_j$ are finite.
Now the functions $F$ satisfying~\eqref{ETPq} are the homotheties $F(\bt) = c t_j$ for a single $j \in
[p]$, with $c > 0$ and $k_j \geq \mathcal{N}$.
\end{enumerate}
\end{theorem}

As mentioned above, the special case of univariate preservers (i.e., $p=1, l= k_1$) was worked out in \cite{BGKP-TN}. Also, observe in both results that for $\mathcal{N} = 2$ (i.e., when considering only $1 \times 1$ and $2 \times 2$ minors), the transforms are products of power functions and indicators in multiple variables. In contrast, for $\mathcal{N} > 2$ it is striking that only one variable shows up -- and that indicators do not.

\subsection{Transforms of symmetric kernel-tuples}

We next turn to the counterparts of the above results for transforms of \textit{symmetric} kernels on a totally ordered set $X$. Already the results are more involved even for $\mathcal{N}=2$.

\begin{definition}\hfill\label{Dstrict}
\begin{enumerate}
\item Given a nonempty totally ordered set $X$, and an integer $1 \leq k \leq \infty$, define $\STN_X^{(k)}$ to be the set of kernels $K : X \times X \to [0,\infty)$ that are $\TN^{(k)}$ and symmetric: $K(x,x') = K(x',x)$ for all $x,x' \in X$. We also write $\STN_X$ for $\STN^{(\infty)}_X$.

\item Similarly, one has the ``totally positive'' counterparts $\STP_X^{(k)}$, including $\STP_X = \STP_X^{(\infty)}$.

\item We will also use the following two notions. First, a function $F : A \to [0,\infty)$, where $A \subseteq [0,\infty)^p$, is \textit{multiplicatively mid-convex} if 
\[
F(t_1, \dots, t_p) F(t'_1, \dots, t'_p) \geq F \left( \sqrt{t_1 t'_1}, \dots, \sqrt{t_p t'_p} \right)^2,
\]
whenever all three arguments lie in $A$.

\item Finally, we say a function $f: A \subseteq \mathbb{R}^{p} \to \mathbb{R}$ is \textit{jointly non-decreasing} (respectively, \textit{jointly increasing}) if 
\[
\bx,\by \in A,\ x_j < y_j\ \forall j \implies f(\bx) \leq f(\by) \text{ (resp.\ } f(\bx) < f(\by)\text{)}.
\]
\end{enumerate}
\end{definition}

Now it is natural to ask the symmetric counterpart of Question~\ref{Qmain}:

\begin{question}
Given a nonempty totally ordered set $X$, an integer $p \geq 1$, and
arbitrary integers $1 \leq k_1, \dots, k_p, l \leq \infty$, classify all
entrywise transforms
\begin{equation}\label{ESTNq}
F[-] : \STN^{(k_1)}_X \times \cdots \times \STN^{(k_p)}_X \to \STN^{(l)}_X,
\end{equation}
and similarly, all entrywise transforms
\begin{equation}\label{ESTPq}
F[-] : \STP^{(k_1)}_X \times \cdots \times \STP^{(k_p)}_X \to \STP^{(l)}_X.
\end{equation}
\end{question}

Our main results are as follows.

\begin{theorem}[Symmetric $\TN$ case]\label{TmainSTN}
Fix a nonempty totally ordered set $X$ and an integer $p \geq 1$. Now
fix integers $1 \leq k_1, \dots, k_p, l \leq \infty$, and define
$\mathcal{N} := \min \{ |X|, l \} \in \mathbb{Z}_{>0} \sqcup \{\infty\}$.
\begin{enumerate}
\item If $\mathcal{N} = 1$, then the maps $F$
satisfying~\eqref{ESTNq} are all functions $: [0,\infty)^p \to
[0,\infty)$.

\item If $\mathcal{N} = 2$, then the functions $F : [0,\infty)^p \to
[0,\infty)$ satisfying~\eqref{ESTNq} admit the following characterization:
\begin{enumerate}
\item If $|X|=2$, then such functions are precisely the nonnegative, jointly non-decreasing, and multiplicatively mid-convex maps on $[0,\infty)^p$, which do not depend on any $t_j$ for which $k_j < \mathcal{N} = 2$. In particular, each such $F$ is continuous on $(0,\infty)^p$.

\item Otherwise such functions are precisely the transforms in Theorem~\ref{TmainTN}(2).
\end{enumerate}

\item If $\mathcal{N} = 3$, then the functions $F$
satisfying~\eqref{ESTNq} are precisely the nonnegative constants and the mixed
power functions $F(\bt) = c \prod_{j \in J} t_j^{\alpha_j}$ with $c >
0$, $J$ a nonempty subset of $[p]$, and $\alpha_j \geq 1$, $k_j \geq \mathcal{N}$ for all $j \in J$.

\item If $\mathcal{N} = 4$, then the functions $F$ satisfying~\eqref{ESTNq} are precisely the nonnegative constants and the individual power functions $F(\bt) = c t_j^{\alpha_j}$ with $c > 0$, $k_j \geq \mathcal{N}$, and
\begin{itemize}
\item $\alpha_j \in \{1\}\cup [2,\infty)$ if $|X|=4$;
\item $\alpha_j = 1$ (i.e.\ $F$ is an individual homothety) otherwise.
\end{itemize}

\item In all other cases -- i.e.\ if $5 \leq \mathcal{N} \leq \infty$,
the functions $F$ satisfying~\eqref{ESTNq} are precisely the nonnegative constants
and the individual homotheties $F(\bt) = c t_j$ for a single $j \in
[p]$, with $c > 0$ and $k_j \geq \mathcal{N}$.
\end{enumerate}
\end{theorem}

Finally, we state the symmetric $\TP$ counterpart:

\begin{theorem}[Symmetric $\TP$ case]\label{TmainSTP}
Fix a nonempty totally ordered set $X$ and an integer $p \geq 1$. Now
fix integers $1 \leq k_1, \dots, k_p, l \leq \infty$, and define
$\mathcal{N} := \min \{ |X|, l \} \in \mathbb{Z}_{>0} \sqcup \{\infty\}$.
\begin{enumerate}
\item If $\mathcal{N} = 1$, then the maps $F$
satisfying~\eqref{ESTPq} are all functions $: (0,\infty)^p \to
(0,\infty)$.

\item If $\mathcal{N} = 2$, then the functions $F : (0,\infty)^p \to
(0,\infty)$ satisfying~\eqref{ESTPq} admit the following characterization:
\begin{enumerate}
\item If $|X|=2$, then such functions are precisely the positive, jointly increasing, and multiplicatively mid-convex maps on $(0,\infty)^p$, which do not depend on any $t_j$ for which $k_j < \mathcal{N} = 2$. In particular, each such $F$ is continuous on $(0,\infty)^p$.

\item Otherwise such functions are precisely the mixed power functions in Theorem~\ref{TmainTP}(2).
\end{enumerate}

\item If $\mathcal{N} = 3$, then the functions $F$
satisfying~\eqref{ESTPq} are precisely the mixed
power functions $F(\bt) = c \prod_{j \in J} t_j^{\alpha_j}$ with $c >
0$, $J$ a nonempty subset of $[p]$, and $\alpha_j \geq 1,\ k_j \geq \mathcal{N}$ for all $j \in J$.

\item If $\mathcal{N} = 4$, we further assume that if $X$ is infinite then all $k_j < \infty$. Now the maps $F$ satisfying~\eqref{ESTPq} are precisely the individual powers $F(\bt) = c t_j^{\alpha_j}$ with $c > 0$, $k_j \geq \mathcal{N}$, and
\begin{itemize}
\item $\alpha_j \in \{1\}\cup [2,\infty)$ if $|X|=4$;
\item $\alpha_j = 1$ (i.e.\ $F$ is an individual homothety) otherwise.
\end{itemize}

\item In all other cases -- i.e.\ if $5 \leq \mathcal{N} \leq \infty$ -- we again assume that if $X$ is infinite but $l<\infty$ then all $k_j < \infty$. Now the functions $F$ satisfying~\eqref{ESTPq} are precisely the individual homotheties $F(\bt) = c t_j$ for a single $j \in [p]$, with $c > 0$ and $k_j \geq \mathcal{N}$.
\end{enumerate}
\end{theorem}

Each of the next four sections proves one theorem above.

\section{Totally nonnegative kernel transforms}

In this section, we characterize the total nonnegativity transforms by
showing Theorem~\ref{TmainTN}. We begin with a simple observation.
If $X' \subseteq X, Y' \subseteq Y$ are nonempty subsets, and $K : X' \times Y' \to
[0,\infty)$ is a kernel, define its extension $\widetilde{K}$ to $X
\times Y$ via ``padding by zeros'':
\begin{equation}\label{Epadding}
    \widetilde{K}(x,y) := \begin{cases}
K(x,y), \qquad &\text{if } (x,y) \in X' \times Y',\\
0, &\text{otherwise.}
\end{cases}
\end{equation}

\begin{lemma}\label{Lpadding}
Suppose $K : X' \times Y' \to [0,\infty)$ is $\TN^{(k)}$ for some $1 \leq
k \leq \infty$. Then so is $\widetilde{K}$. In particular, the padding
map $K \mapsto \widetilde{K}$ (injectively) embeds $\TN_{X' \times
Y'}^{(k)}$ into $\TN_{X \times Y}^{(k)}$. Similar statements hold if $Y' = X'$, $Y=X$, and $K$ is symmetric: then so is $\widetilde{K}$ and $\STN^{(k)}_{X'} \hookrightarrow \STN^{(k)}_X$.
\end{lemma}

\begin{proof}
The first assertion holds because all square matrices $\widetilde{K}[
\bx; \by ]$ are either of the form $K[ \bx; \by ]$, or have a row or
column of zeros. The second statement follows from the first. In turn, these impliy their symmetric counterparts, by embedding $K$ into $\widetilde{K}$ as a ``principal submatrix''.
\end{proof}

We now proceed. The main novel step in this section is to show that in the ``non-singleton cases'', $F$ is a product of power functions and indicators of the positive semi-axis, up to scaling.

\begin{theorem}\label{Tpowers}
In the setting of Theorem~\ref{TmainTN}, if $2 \leq \mathcal{N} = \min\{
|X|, |Y|, l \} \leq \infty$, then every function satisfying~\eqref{ETNq}
is of the form
\[
F({\bf t}) = c \prod_{j \in J} t_j^{\alpha_j} \prod_{i \not\in J} {\bf 1}_{t_i > 0} \quad \text{with}\ c, \alpha_j\geq 0, \text{and}\ J\subseteq [p].
\]
\end{theorem}

\begin{proof}
Choose any $X' \subseteq X$ and $Y' \subseteq Y$ of size $2$ each. 
If $F[-]$ satisfies~\eqref{ETNq} then it follows via Lemma~\ref{Lpadding} that
\[
F[-] : \TN_{2 \times 2} \times \cdots \times \TN_{2 \times 2} \to
\TN_{2 \times 2}.
\]

In particular, if one fixes all but one arguments $t_1, \dots, \widehat{t_j}, \dots, t_p > 0$, then the function
\begin{equation}\label{Egj}
g_j(t) := F(t_1, \dots, t_{j-1}, \ t, \ t_{j+1}, \dots, t_p)
\end{equation}
when applied entrywise, preserves $\TN_{2 \times 2}$ (the other matrices in applying $F[-]$ are $t_j {\bf 1}_{2 \times 2}$). In particular, by \cite[Theorem~3.3]{BGKP-TN}, the restriction $g_j|_{(0,\infty)}$ is a power function up to scaling, say $g_j(t) = c_j t^{\alpha_j}$ for some $c_j = c_j(t_1, \dots, \widehat{t_j}, \dots, t_p),\ \alpha_j = \alpha_j(t_1, \dots, \widehat{t_j}, \dots, t_p) \geq 0$.

We now establish the claim that every such map $F$ is of the required form, proceeding in steps.\medskip

\noindent
\textbf{Step 1:} We first show that $F$ is as asserted on the interior of the orthant:

\begin{lemma}\label{Lclaim}
Let $p \geq 1$ and a function $F : (0,\infty)^p \to \mathbb{R}$ not be identically zero. Suppose that for each fixed $t_1, \dots, \widehat{t_j}, \dots, t_p > 0$, the map $g_j|_{(0,\infty)}$ in~\eqref{Egj} is a power function $ c_j t^{\alpha_j}$ for some $c_j = c_j(t_1, \dots, \widehat{t_j}, \dots, t_p),\ \alpha_j = \alpha_j(t_1, \dots, \widehat{t_j}, \dots, t_p) \geq 0$. Then $F(\bt) = c \bt^{\boldsymbol \alpha}$ is a mixed power function (up to scaling) for all $t_j > 0\ \forall j$, for some exponents ${\boldsymbol \alpha} \geq {\bf 0}$ and scalar $c = F(1,\dots,1) > 0$.
\end{lemma}

In other words, the maps $c_j, \alpha_j$ are constant for each $j \in [p]$, and $c_1 = \cdots = c_p$.

\begin{remark}
As the proof shows,
(a)~this result is significantly involved than the $p=1$ case (more precisely, than the derivation of the $p=1$ case from $\TN_{2 \times 2}$ preservation).
Moreover,
(b)~the non-negativity -- more precisely, the invariance of signs -- of the exponents $\alpha_j$ is crucial. Otherwise the result already fails for $p=2$ -- because the conditions (for every choice of fixed scalars $x_0, y_0 > 0$)
\[
F(x,y_0) = c_{y_0} x^{\alpha_{y_0}}, \quad F(x_0,y) = d_{x_0} y^{\beta_{x_0}}, \qquad \forall x,y>0
\]
are satisfied by a larger family than the mixed power functions $c x^\alpha y^\beta$:
\[
F(x,y) = \exp \left[ K + \alpha_1 \log(x) + \beta_1 \log(y) + \gamma \log(x) \log(y) \right].
\]
Indeed, in this case one has:
\[
c_{y_0} = e^K y_0^{\beta_1}, \quad \alpha_{y_0} = \alpha_1 + \gamma \log(y_0); \qquad
d_{x_0} = e^K x_0^{\alpha_1}, \quad \beta_{x_0} = \beta_1 + \gamma \log(x_0).
\]
\end{remark}

\begin{proof}[Proof of Lemma~\ref{Lclaim}]
We proceed by (strong) induction on the number of variables $p$. The base case of $p = 1$ is a tautology. The $p=2$ case is used in the general induction step, so we show it next. For each $t_2 > 0$, by the $p=1$ case $F(t_1, t_2) = c(t_2) t_1^{\alpha_{t_2}}$ for all $t_1, t_2 > 0$, for some scalars $c(t_2), \alpha_{t_2} \geq 0$. Setting $t_1 = 1$ yields $c(t_2) = F(1,t_2)$. Switching the roles of $t_1$ and $t_2$, we similarly obtain:
\begin{equation}\label{Ep20}
F(1,t_2) t_1^{\alpha_{t_2}} = F(t_1, t_2) = F(t_1,1) t_2^{\beta_{t_1}}, \qquad \forall t_1, t_2 > 0
\end{equation}
for a scalar $\beta_{t_1} \geq 0$. Setting $t_2= 1$ in the first equality and $t_1 = 1$ in the second yields:
\[
F(t_1,1) = F(1,1) t_1^{\alpha_1}, \qquad
F(1,t_2) = F(1,1) t_2^{\beta_1},
\]
which combined with the preceding equation yields:
\begin{equation}\label{Ep2}
F(t_1,t_2) = F(1,1) t_1^{\alpha_1} t_2^{\beta_{t_1}} = F(1,1) t_2^{\beta_1} t_1^{\alpha_{t_2}}, \qquad \forall t_1, t_2 > 0.
\end{equation}

Now if $F(1,1) = 0$ then $F \equiv 0$, which contradicts the hypotheses. Else divide by $F(1,1)$ and take logs to obtain
\[
\frac{\beta_{t_1} - \beta_1}{\log t_1} = \frac{\alpha_{t_2} - \alpha_1}{\log t_2}, \qquad \forall t_1, t_2 \in (0,\infty) \setminus \{ 1 \}.
\]
Since the LHS and RHS are functions of $t_1, t_2$ individually, both sides must equal a constant, say $K$. Thus,
\[
\beta_{t_1} = \beta_1 + K \log t_1, \qquad \alpha_{t_2} = \alpha_1 + K \log t_2, \qquad \forall t_1, t_2 > 0.
\]
At the same time, the hypotheses imply $\beta_{t_1}, \alpha_{t_2} \geq 0$ for all $t_1, t_2 > 0$. This is possible only if $K = 0$, i.e.\ $\beta_{t_1} \equiv \beta_1, \alpha_{t_2} \equiv \alpha_1$. Now~\eqref{Ep2} yields Lemma~\ref{Lclaim} for $p=2$:
\[
F(t_1, t_2) = F(1,1) t_1^{\alpha_1} t_2^{\beta_1}, \qquad \forall t_1, t_2 > 0.
\]

We now proceed to the induction step for $F : (0,\infty)^p \to (0,\infty)$; thus, we know the result for functions of $1, 2, \dots, p-1$ variables. By the induction hypothesis, for each $t \in (0,\infty)$, the functions
\begin{align*}
& f_1 : (s_1, \dots, s_{p-1}) \mapsto F(t,\ s_1, \dots, s_{p-1})
& f_p : (s_1, \dots, s_{p-1}) \mapsto F(s_1, \dots, s_{p-1}, \ t)
\end{align*}
must be of the form
\[
f_1({\bf s}) = d(t) \prod_{1 < j \leq p} s_j^{\beta_j(t)}, \qquad
f_p({\bf s}) = c(t) \prod_{1 \leq j < p} s_j^{\alpha_j(t)}, \qquad \forall {\bf s} \in (0,\infty)^{p-1}, \ t > 0
\]
for some functions $c, d, \alpha_j, \beta_j \geq 0$. As above
$d(t) = F(t, {\bf 1}_{p-1})$ and
$c(t) = F({\bf 1}_{p-1}, t)$, so akin to~\eqref{Ep2}:
\begin{equation}\label{Ep}
F(t_1, {\bf 1}_{p-1}) \prod_{1 < j \leq p} t_j^{\beta_j(t_1)} = F(\bt) = 
F({\bf 1}_{p-1}, t_p) \prod_{1 \leq j < p} t_j^{\alpha_j(t_p)}, \qquad \forall \bt \in (0,\infty)^p. 
\end{equation}

The first claim is that the ``constants'' $F(t_1, {\bf 1}_{p-1}), F({\bf 1}_{p-1}, t_p)$ never vanish. Indeed, if say $F({\bf 1}_{p-1}, t_p) = 0$ for some $t_p>0$, then for every $t_1>0$, \eqref{Ep} implies $F(t_1, {\bf 1}_{p-1}) = 0$, whence again by~\eqref{Ep}, $F \equiv 0$ (which contradicts the hypotheses).

But now we can rewrite~\eqref{Ep} as:
\[
\frac{F(t_1, {\bf 1}_{p-1})}{F({\bf 1}_{p-1}, t_p)} t_1^{-\alpha_1(t_p)} t_p^{\beta_p(t_1)} = \prod_{1<j<p} t_j^{\alpha_j(t_p) - \beta_j(t_1)}, \qquad \forall \bt \in (0,\infty)^p.
\]
As the left-hand side does not depend on $t_2, \dots, t_{p-1}$, it follows that $\alpha_j(t_p) = \beta_j(t_1) =: \alpha_j$, say, for all scalars $t_1, t_p > 0$ and indices $1<j<p$. Thus,~\eqref{Ep} reduces to
\begin{equation}\label{Epp}
F(t_1, {\bf 1}_{p-1}) t_p^{\beta_p(t_1)} \prod_{1 < j < p} t_j^{\alpha_j} = F(\bt) = 
F({\bf 1}_{p-1}, t_p) t_1^{\alpha_1(t_p)} \prod_{1 < j < p} t_j^{\alpha_j},
\qquad \forall \bt \in (0,\infty)^p.
\end{equation}

Finally, define $g(t_1, t_p) := F(t_1, {\bf 1}_{p-2}, t_p)$ on $(0,\infty)^2$; then the preceding equation becomes
\[
g(1,t_p) t_1^{\alpha_1(t_p)} = g(t_1, t_p) = g(t_1,1) t_p^{\beta_p(t_1)}, \qquad \forall t_1, t_p > 0
\]
with $\alpha_1, \beta_p \geq 0$. But this is precisely~\eqref{Ep20}; now the analysis for $p=2$ yields
\[
g(t_1,t_p) = g(1,1) t_1^{\alpha_1} t_p^{\beta_p}, \qquad \forall t_1, t_p > 0.
\]
Finally, using these last two equations and~\eqref{Epp}, we have:
\begin{align*}
F(\bt) = &\ g(1,t_p) t_1^{\alpha_1(t_p)} \prod_{1<j<p} t_j^{\alpha_j} = g(t_1, t_p) \prod_{1<j<p} t_j^{\alpha_j} = g(1,1) t_1^{\alpha_1} t_p^{\beta_p} \prod_{1<j<p} t_j^{\alpha_j}\\
= &\ F({\bf 1}_p) \prod_{j=1}^p t_j^{\alpha_j}, \qquad \forall \bt \in (0,\infty)^p \quad (\text{where }
\alpha_p := \beta_p). \qedhere
\end{align*}
\end{proof}

\noindent \textbf{Step 2:} \textit{Extending the positive powers to the boundary.}

One can possibly try to prove the full result also by induction on $p$, but we proceed directly. 
Consider $F$ on $[0,\infty)^p$. If $F = 0$ at a single point in $(0,\infty)^p$ then $F \equiv 0$ on $(0,\infty)^p$ by Lemma~\ref{Lclaim}. But then, given any $\bt \in [0,\infty)^p$, we evaluate $F$ at the $\TN_{2 \times 2}$ matrices
\begin{equation}\label{Estep21}
A_j := \begin{pmatrix} t_j + 1 & t_j \\ t_j & t_j \end{pmatrix}, \qquad j \in [p].
\end{equation}
By the hypotheses, $0 \leq \det F[{\bf A}]  = F(\bt + {\bf 1}_p) F(\bt) - F(\bt)^2 = - F(\bt)^2$, and so $F \equiv 0$ on $[0,\infty)^p$.

Thus, we henceforth suppose $F(\bt) = c \bt^{\boldsymbol \alpha}$ on $(0,\infty)^p$, with $c > 0$ and ${\boldsymbol \alpha} \geq {\bf 0}$.
By Lemma~\ref{Lclaim}, partition $[p] = J_0 \sqcup J_+$, with
\begin{equation}
J_0 := \{ j \in [p] : \alpha_j = 0 \}, \qquad J_+ := \{ j \in [p] : \alpha_j > 0 \}.
\end{equation}

We need to compute $F$ on the boundary of the orthant $[0,\infty)^p$. We begin by showing that $F$ vanishes whenever $J_+ \neq \emptyset$ and $t_{j_0} = 0$ for some $j_0 \in J_+$ (notice, this is consistent with the assertion of the theorem). To show this, assume without loss of generality that $j_0 = 1$ (for expositional convenience), and choose any $t_j \in [0,\infty)$ for all $j>1$. Now evaluate $F$ at the sequence of $\TN_{2 \times 2}$ matrices
\begin{equation}\label{Estep22}
A_1^{(n)} := \frac{1}{n} \Id_2, \qquad A_j^{(n)} := \frac{1}{n} \Id_2 + t_j {\bf 1}_{2 \times 2}\ \ (j>1).
\end{equation}
Then with $t_1 = 0$,
\[
0 \leq \det F[{\bf A}^{(n)}] = F(\bt + \frac{1}{n} {\bf 1}_p)^2 - F(\bt)^2 = c^2 n^{-2 \alpha_1} \prod_{j>1} (t_j + \frac{1}{n})^{2\alpha_j} - F(\bt)^2, \qquad \forall n \geq 1.
\]
Letting $n \to \infty$ and since $\alpha_1 > 0$, we obtain $F(\bt) = 0$ if $j_0 \in J_+$ and $t_{j_0} = 0$.\medskip

\noindent \textbf{Step 3:} \textit{Extending the zero powers to the boundary.}

It remains to compute $F$ at all other points (regardless of whether or not $J_+$ is empty), i.e.\ at
\[
\bt' = (t_1, \dots, t_p) \in (0,\infty)^{J_+} \times \left( [0,\infty)^{J_0} \setminus (0,\infty)^{J_0} \right).
\]
(In particular, we are already done if $J_0$ is empty, so henceforth we suppose $J_0 \neq \emptyset$.)

We begin by noting that for each $j_0 \in J_0$, the function
\[
g_{j_0} : [0,\infty) \to [0,\infty); \qquad
s \mapsto F({\bf 1}_{j_0-1}, \ s, \ {\bf 1}_{p-j_0})
\]
entrywise preserves $\TN_{2 \times 2}$ and is of the form $c \cdot s^0$ for $s>0$, by Lemma~\ref{Lclaim}. Hence by the $p=1$ case in \cite[Theorem~3.3]{BGKP-TN}, the scalar $\delta_{j_0} := c^{-1} g_{j_0}(0)$ equals $0$ or $1$, depending on if $g_{j_0}$ is Heaviside or constant, respectively.
Setting $J_{00} := \{ j_0 \in J_0 : \delta_{j_0} = 0 \}$, it suffices to show that
\begin{equation}\label{Etoshow}
F(\bt') = c \cdot \prod_{j \in J_+} t_j^{\alpha_j} \cdot \prod_{j_0 \in J_{00}\, :\, t_{j_0} = 0} \delta_{j_0}, \qquad
\forall \bt' \in (0,\infty)^{J_+} \times \left( [0,\infty)^{J_0} \setminus (0,\infty)^{J_0} \right).
\end{equation}

Noting that the right-hand side vanishes if and only if the second product is not an empty one, we work in two cases.

\begin{enumerate}
\item First assume there is at least one zero entry among $\{ t_j : j \in J_{00} \}$ -- say there are precisely $l>0$ such entries, indexed by $j_1, \dots, j_l \in J_{00}$. We need to show that for all
\[
{\bf r} = (r_a) \in (0,\infty)^{J_{00} \setminus \{ j_1, \dots, j_l \}}, \qquad
{\bf s} = (s_b) \in [0,\infty)^{J_0 \setminus J_{00}}, \qquad
{\bf t}'' = (t_d) \in (0,\infty)^{J_+},
\]
and writing the arguments of $F$ in the following re-indexed order:
\[
j_1, \dots, j_l; \quad J_{00} \setminus \{ j_1, \dots, j_l \}; \quad J_0 \setminus J_{00}; \quad J_+,
\]
we have that
$F({\bf 0}_l; \ {\bf r}; \ {\bf s}; \ {\bf t}'') = 0$.

For this, choose the $\TN_{2 \times 2}$ matrices
\begin{align}\label{Estep31}
\begin{aligned}
A_{j_1} := &\ \begin{pmatrix} 1 & 0 \\ 1 & 0 \end{pmatrix}; \qquad
A_{j_2}  = \cdots = A_{j_l} := \begin{pmatrix} 1 & 0 \\ 1 & 1 \end{pmatrix};\\
A_a := &\ \begin{pmatrix} r_a & r_a \\ 1 & 1 \end{pmatrix}, \ a \in J_{00} \setminus \{ j_1, \dots, j_l \}; \\
A_b := &\ \begin{pmatrix} s_b & s_b \\ 1 & 1 \end{pmatrix},\ b \in J_0 \setminus J_{00}; \\
A_d := &\ \begin{pmatrix} t_d & t_d \\ 1 & 1 \end{pmatrix}, \ d \in J_+.
\end{aligned}
\end{align}
Then $F[ {\bf A} ] \in \TN_{2 \times 2}$, so since $j_1 \in J_{00}$,
\[
0 \leq \det F[ {\bf A} ] = F({\bf 1}_l; {\bf r}; {\bf s}; {\bf t}'') F(0, {\bf 1}_{p-1}) - F({\bf 0}_l; {\bf r}; {\bf s}; {\bf t}'') F({\bf 1}_p) = - c F({\bf 0}_l; {\bf r}; {\bf s}; {\bf t}'').
\]
Since $F \geq 0$, this implies $F({\bf 0}_l; {\bf r}; {\bf s}; {\bf t}'') = 0$, which concludes the proof in this case.

\item Finally, suppose $t_{j_0} > 0$ for all $j_0 \in J_{00}$. For expositional convenience, we also assume in this case, without loss of generality, that $J_0 \setminus J_{00} = \{ 1, \dots, p_0 \}$ for some $1 \leq p_0 \leq p$. We are to show that
$F(\bt') = c \prod_{j \in J_+} t_j^{\alpha_j}$. By relabeling coordinates if necessary, we are to show
\begin{equation}\label{Eclaim1}
F({\bf 0}_l, \bt'') = c \prod_{j \in J_+} t_j^{\alpha_j}, \qquad \forall 0 \leq l \leq p_0, \quad \forall t_j > 0, j > l
\end{equation}
where $\bt'' := (t_{l+1}, \dots, t_p)$.

We proceed by induction on $l \geq 0$, with the base case of $l=0$ shown in Lemma~\ref{Lclaim}. We next show the $l=1$ case: consider the $\TN_{2 \times 2}$ matrices
\begin{equation}\label{Estep321}
A_1 := \begin{pmatrix} 1 & 0 \\ 1 & 0 \end{pmatrix}, \qquad
A_j := \begin{pmatrix} t_j & t_j \\ 1 & 1 \end{pmatrix},\ j>1.
\end{equation}

Then $F[ {\bf A} ] \in \TN_{2 \times 2}$, so using the $l=0$ case and that $J_{00} = \emptyset$,
\[
0 \leq \det F[ {\bf A} ] = F(1, \bt'') F(0, {\bf 1}_{p-1}) - F(0, \bt'') F({\bf 1}_p) = c^2 \prod_{j \in J_+} t_j^{\alpha_j} - c F(0,\bt'').
\]
If we instead change $A_1 := \begin{pmatrix} 0 & 1 \\ 0 & 1 \end{pmatrix}$, then we similarly get
\[
0 \leq \det F[ {\bf A} ] = c F(0,\bt'') - c^2 \prod_{j \in J_+} t_j^{\alpha_j}.
\]
These together imply the $l=1$ case.

Finally, for the induction step, consider the $\TN_{2 \times 2}$ matrices
\begin{equation}\label{Estep322}
A_1 := \begin{pmatrix} 1 & 0 \\ 1 & 0 \end{pmatrix}; \qquad
A_2  = \cdots = A_l := \begin{pmatrix} 1 & 1 \\ 0 & 0 \end{pmatrix}; \qquad
A_j := t_j {\bf 1}_{2 \times 2},\ j>l.
\end{equation}
Proceeding as above, and using the induction step as well as the $l=0,1$ cases,
\begin{align*}
0 \leq \det F[{\bf A}] = &\ F({\bf 1}_l, \bt'') F({\bf 0}_l, \bt'') - F(0, {\bf 1}_{l-1}, \bt'') F(1, {\bf 0}_{l-1}, \bt'')\\
= &\ c \prod_{j \in J_+} t_j^{\alpha_j} \left( F({\bf 0}_l, \bt'') - c \prod_{j \in J_+} t_j^{\alpha_j} \right).
\end{align*}
We also get the reversed inequality by exchanging the columns of $A_1$. This shows the induction step, and concludes the proof. \qedhere
\end{enumerate}
\end{proof}

With Theorem~\ref{Tpowers} in hand, we can proceed.

\begin{proof}[Proof of Theorem~\ref{TmainTN}]
The result is obvious if $\mathcal{N} = 1$, as one only works with
$1 \times 1$ minors -- i.e.\ function values.
\begin{enumerate}
\setcounter{enumi}{1}
\item Suppose $\mathcal{N} = 2$. By Theorem~\ref{Tpowers}, every
preserver must be of the form
\[
F({\bf t}) = c \prod_{j \in J} t_j^{\alpha_j} \prod_{i \not\in J} {\bf 1}_{t_i > 0}, 
\]
with all $\alpha_j \geq 0$ and $J\subseteq [p]$. Moreover, suppose $c>0$ and some $k_{j_0} = 1$.
Fix points $x_1 < x_2$ in $X$ and $y_1 < y_2$ in $Y$, and consider the
test kernels evaluated at these pairs of points to be the matrices
\[
K_{j_0} = \begin{pmatrix} 0 & 1 \\ 1 & 0 \end{pmatrix}, \qquad
K_j = \begin{pmatrix} 1 & 1 \\ 1 & 1 \end{pmatrix} \ \forall j \neq j_0.
\]
Let $\widetilde{K}_j$ be the respective paddings by $0$ as in~\eqref{Epadding}; by
Lemma~\ref{Lpadding}, $\widetilde{K}_{j_0} \in \TN_{X \times Y}^{(1)}
\setminus \TN_{X \times Y}^{(2)}$ and $\widetilde{K}_j \in \TN_{X \times
Y}^{(\infty)}$ for all other $j$. We see that $j_{0}\in J$. Otherwise, from the hypothesis that $F \circ {\bf \widetilde{K}}$ is $\TN^{(2)}$, we get $(F \circ {\bf \widetilde{K}})[ \bx; \by ] = K_{j_0}\in \TN^{(2)}$, which is clearly absurd. Therefore, $j_{0}\in J$ and in that case 
\[
0 \leq \det (F \circ {\bf \widetilde{K}})[ \bx; \by ] =
c^2(0^{2\alpha_{j_0}} - 1).
\]
This implies $\alpha_{j_0} = 0$, as desired.

We take a short detour to prove that $\TN_{2\times 2}$ is invariant under the Schur product. Indeed, if we let $A= (a_{ij})_{i,j=1}^{2},B = (b_{ij})_{i,j=1}^{2}$ be any $\TN_{2\times 2}$ matrices then the inequality
$a_{11}b_{11}a_{22}b_{22}\geq a_{12}b_{12}a_{21}b_{21}$ holds, since $a_{11}a_{22} \geq a_{12}a_{21}$ and $b_{11}b_{22} \geq b_{12}b_{21}$, by assumption. Therefore,  $A\circ B \in \TN_{2\times 2}$.

Returning to the proof, conversely (still with $\mathcal{N} = 2$) suppose $F$ is of the given
form, and $K_j \in \TN_{X \times Y}^{(k_j)}$ for all $j$. As $F \geq 0$,
$F \circ {\bf K}$ is $\TN^{(1)}$. Now fix $\bx \in \inc{X}{2}$ and $\by
\in \inc{Y}{2}$; if $k_j \geq 2$ then $K_{j}(x_1,y_1) K_{j}(x_2,y_2) \geq
K_{j}(x_1,y_2) K_{j}(x_2,y_1)$. Raising both sides to the $\alpha_j$th power, and
multiplying over all $j$ with $k_j \geq 2$ (and also by $c \geq 0$), we
see that $\det K_{j}^{\circ{\alpha_j}}[ \bx; \by ] \geq 0$. Moreover, ${\bf 1}_{t_{j}>0}[(K(x_{i}, y_{j}))_{i,j=1}^{2}] \in \TN$. Combining this with the fact that $\TN_{2\times 2}$ is invariant under the Schur product concludes the
proof for $\mathcal{N} = 2$, except the case when all $k_j = 1$. But in
this case $F$ of the claimed form is simply a nonnegative constant map,
and this sends ${\bf K}$ to a $\TN$ kernel.

\item Now let $\mathcal{N} \geq 3$ and $F$ be nonconstant. As above, by
Theorem~\ref{Tpowers} $F$ is of the form

\begin{equation}\label{Eindicator}
F({\bf t}) = c \prod_{j \in J} t_j^{\alpha_j} \prod_{i \not\in J} {\bf 1}_{t_i > 0} 
\end{equation}
with all $\alpha_j \geq 0$ and $J\subseteq [p]$. Suppose for
contradiction that at least two powers are positive, say $\alpha_1,
\alpha_2 > 0$ (and $c>0$). Fix triples $\bx \in \inc{X}{3}$ and $\by \in
\inc{Y}{3}$, and consider the test kernels evaluated at these triples to
be the $\TN$ matrices
\[
K_1 = \begin{pmatrix} 1 & 1 & 0 \\ 1 & 1 & 1 \\ 1 & 1 & 1 \end{pmatrix},
\qquad K_2 = K_1^T, \qquad K_j = \begin{pmatrix} 1 & 1 & 1 \\ 1 & 1 & 1
\\ 1 & 1 & 1 \end{pmatrix} \ \forall 3 \leq j \leq p.
\]
(This example is taken from \cite{FJS}.)
Let $\widetilde{K}_j$ be the respective paddings by $0$; by
Lemma~\ref{Lpadding}, $\widetilde{K}_j \in \TN_{X \times Y}^{(\infty)}$
for all $j$. By hypothesis, $F \circ {\bf \widetilde{K}}$ is $\TN^{(3)}$,
and so
\begin{equation}\label{EFJS}
0 \leq \det (F \circ {\bf K})[ \bx; \by ] = \det c \begin{pmatrix}
1 & 1 & 0 \\ 1 & 1 & 1 \\ 0 & 1 & 1 \end{pmatrix} = -c^3.
\end{equation}
This contradiction shows that there is at most one $j_0 \in [p]$ with $\alpha_{j_0} >
0$, so that the dependence on the other variables $t_j$ in~\eqref{Eindicator} is through their indicator ${\bf 1}_{t_j > 0}$ or their zeroth power. We show that it is never the former, which proves that $J = [p]$, such a $j_0$ exists, and all other $\alpha_j = 0$ (as $F$ is nonconstant). Thus, suppose to the contrary that $i\in J^{c}$, and let 
\begin{equation}\label{Esqrt2}
K_{i} = \begin{pmatrix}
1 & 1/\sqrt{2} & 0 \\ 1/\sqrt{2} & 1 & 1/\sqrt{2} \\ 0 & 1/\sqrt{2} & 1
\end{pmatrix}, \qquad K_j = \begin{pmatrix} 1 & 1 & 1 \\ 1 & 1 & 1
\\ 1 & 1 & 1 \end{pmatrix} \ \forall  j \neq i
\end{equation}
be $\TN$ kernels evaluated at $[\bx; \by]$.
  Applying ${\bf 1}_{t_i>0}$ to the padding $\widetilde{K}_i$ (equivalently, applying $F[-]$ to $\widetilde{\bf K}$) yields at $[\bx; \by]$  the matrix in \eqref{EFJS} with negative determinant. Hence $J^{c}=\emptyset$, which yields (as above) a unique $j_0$ such that $\alpha_{j_0}>0$.\footnote{We take the occasion to point out a notational typo in the corresponding univariate proof in \cite{BGKP-TN}. Specifically, on pp.~96 of \textit{loc.\ cit.}, in the paragraph before Equation~(3.5), the authors ``conclude by [the $2 \times 2$ case, via padding by zeros] that [the non-constant map] $F(x) = c x^\alpha$ for some $c > 0$ and $\alpha \geq 0$.'' Here, since $F$ is non-constant, when the authors wrote $c x^0$ (i.e.\ $\alpha=0$) they meant the Heaviside indicator $c {\bf 1}_{x>0}$.
 
The same notational typo occurs in \cite[Theorem~3.7(3)(b)]{BGKP-TN}: the authors mean the zeroth power not to be the constant function $1$, but the Heaviside indicator ${\bf 1}_{x>0}$. The reason both of these are notational typos is because the authors declare in the lines preceding \cite[Theorem~3.3]{BGKP-TN} that $0^0 := 1$ -- so $x^0$ should then be ${\bf 1}_{x \geq 0}$, but the authors take it to mean ${\bf 1}_{x>0}$ in the two instances cited.}

We next claim $k_{j_0} \geq \mathcal{N} = 3$. 
Suppose not; then we work with the test kernels evaluated at $\bx, \by$
to be $K_j = {\bf 1}_{3 \times 3}$ for $j \neq j_0$, and $K_{j_0}$ to be the 0-1 matrix in~\eqref{EFJS} (without the $c$).
By Lemma~\ref{Lpadding}, the paddings $\widetilde{K}_j \in \TN_{X \times
Y}^{(k_j)}$ for all $j$, so $F \circ {\bf \widetilde{K}}$ is $\TN^{(3)}$.
But then
\[
0 \leq \det (F \circ {\bf \widetilde{K}})[ \bx; \by ] = \det c K_{j_0} =
-c^3,
\]
which is false, and yields $k_{j_0} \geq \mathcal{N}$. Finally, let the test kernel $K_{j_0}$ evaluated at $\bx, \by$
be the $\TN$ matrix denoted in \eqref{Esqrt2} by $K_i$. As computed explicitly in \cite[Section~3]{BGKP-TN}, $\det K_{j_0}[ \bx; \by]^{\circ \alpha} = 1 - 2^{1-\alpha} < 0$ for $\alpha \in (0,1)$. Thus $\alpha_{j_0} \geq 1$ if $\mathcal{N} \geq 3$.

Conversely, say $\mathcal{N} = 3$ and $F(\bt) = c t_{j_0}^{\alpha_{j_0}}$ is of the desired
form, and $K_{j_0} \in \TN_{X \times Y}^{(k_{j_0})}$. Since we are only
considering $3 \times 3$ minors, it suffices to show that $\TN_{3 \times
3}$ is closed under taking entrywise $\alpha$th powers for $\alpha \geq
1$. But this is precisely \cite[Theorem~5.2]{FJS}.

\item By the preceding case, $F(\bt) = c t_{j_0}^{\alpha_{j_0}}$ for
some $\alpha_{j_0} \geq 1$. Using the padded test kernel
$\widetilde{K}_{j_0}$, where $K_{j_0}$ is the $4 \times 4$ $\TP$ matrix in
\cite[Example 5.8]{FJS} (for suitable $x,\epsilon > 0$), we have $\alpha_{j_0} = 1$.

We next show that $k_{j_0} \geq \mathcal{N}$. Suppose not, and let $X' = Y' = \{  x_1 < x_2 < \cdots < x_{\mathcal{N}} \} \subset (0,\infty)$ if $\mathcal{N}< \infty$, while if $\mathcal{N} = \infty$ we choose $X' = Y'$ to be any increasing sequence $\{ x_{n} : n\geq 1 \}$ of positive real numbers. Now let $K: X' \times Y' \to (0,\infty)$ be the kernel given by the ``Jain matrix'' \cite{Jain, TJ} $(1+ x_i x_j)_{i,j=1}^{\mathcal{N}}$, and $\widetilde{K}$ be its padding by zeros \eqref{Epadding}.

There are two cases. If $\mathcal{N}< \infty$, choose any $\alpha\in (\mathcal{N}-3, \mathcal{N}-2)$; by \cite[Theorem~C]{AK}, it follows that the matrix $K^{\circ \alpha} = K^\alpha[ X'; Y' ]$ is $\TN^{(k_{j_0})}$ but not $\TN_{X\times Y}^{(\mathcal{N})}$. Otherwise $1 \leq k_{j_0}< \infty = \mathcal{N}$; now let $\alpha \in (k_{j_0}-1, k_{j_0})$. Then \cite[Theorem~C]{AK} implies that $K^{\circ \alpha}$ is $\TN^{(k_{j_0})}$ but not $\TN$. In either case, letting the test kernel $K_{j_0} = \widetilde{K}^{\alpha} \in \TN_{X \times Y}^{(k_{j_0})}$ and all other $K_j := {\bf 0}_{X \times Y}$ provides the desired contradiction.

Clearly, constant functions and individual homotheties $F(\bt) = c, ct_{j}$ (when $k_{j}\geq \mathcal{N}$) for $c \geq 0$ send $K_{j} \in \TN_{X\times Y}^{(k_{j})}$ to a $\TN_{X\times Y}^{(l)}$ kernel. This shows the converse. \qedhere
 \end{enumerate}
\end{proof}

\section{Totally positive kernel transforms}

In this section, we prove Theorem~\ref{TmainTP}. We begin by noting some hurdles in proving these assertions via similar strategies to the previous section: First, ``padding by zeros" does not preserve (strict) total positivity, so embedding tools like Lemma~\ref{Lpadding} are lost. To remedy this, we need approximation results like \textit{Whitney's density theorem} \cite{Whitney}, which says that $\TP_{m \times n}^{(p)}$ matrices are dense in $TN_{m \times n}^{(p)}$ matrices for positive integers $m,n,p$. A second subtlety arises when we deal with (infinite) kernels rather than just matrices. For this purpose, we refer to \cite[Theorem 1.2]{BGKP-TN}, which gives the locally uniform approximation of $\TN^{(p)}$ kernels by $\TP^{(p)}$ kernels on the real line and its subsets, via discretized Gaussian convolution:

\begin{theorem}\label{Twhitney2}
Given nonempty subsets $X$ and $Y$ of $\mathbb{R}$, and an integer $k \in [1,\infty)$,
any bounded $\TN^{(k)}$ kernel $K : X \times Y \to \mathbb{R}$ can be
approximated locally uniformly at points of continuity in the interior
of $X \times Y$ by a sequence of $\TP^{(k)}$ kernels on $X \times Y$.
Furthermore, if $Y = X$ and $K$ is symmetric, then the kernels in the approximating sequence may be taken to be symmetric.
\end{theorem}

We now proceed. The first step is to prove the $\TP$ counterpart of Theorem~\ref{Tpowers}.

\begin{theorem}\label{TP-reduction}
 In the setting of Theorem~\ref{TmainTP} (and without the extra assumption in part~(4)), if $2\leq \mathcal{N} = \min\{|X|,|Y|,l\}\leq \infty, $ then every function satisfying~\eqref{ETPq} is of the form $F(\bt) = c {\bt}^{\boldsymbol{\alpha}}= c \prod_{j=1}^{p}t_{j}^{\alpha_{j}},$ with $c > 0$ and all $\alpha_{j} \geq 0$ but ${\boldsymbol \alpha} \neq {\bf 0}$. Moreover, if any $\alpha_j > 0$ then $k_j \geq 2$.
\end{theorem}

This is more involved than Theorem~\ref{Tpowers} because the test matrices there are not available to us to directly obtain the form of $F$. Thus, we start by (a)~proving the continuity and monotonicity (joint and coordinatewise) of every preserver $F[-]: (\TP_{2\times 2})^{p} \to \TP_{2\times2}$, so that we can use the previous section on $\TN$ preservers via approximation results. This is followed by (b)~understanding the domains $X,Y$ on which such kernels can exist.

\begin{proposition}\label{prop-TP continuity}
Let $F:(0,\infty)^{p}\to (0,\infty)$ be a function which maps $(\TP_{2\times 2})^{p}$ into $\TP_{2\times2}.$ Then $F$ is jointly increasing (see Definition~\ref{Dstrict}) and is continuous.
\end{proposition}

\begin{proof}
For joint monotonicity, fix $\bt = (t_{1},\ldots,t_{p}),\ \bt' = (t'_{1},\ldots,t'_{p}) \in (0, \infty)^{p}$ such that $t'_j > t_j$ for all $j \in [p]$. Now consider the following $\TP_{2\times 2}$ matrices:
\[
K_{j} =
\begin{pmatrix} t'_j & t_j \\ t_j & t'_j \end{pmatrix},
\qquad \forall 1\leq j\leq p.
\]
By hypothesis, $F[{\bf K}] \in \TP_{2\times 2}$. So
$0< \det(F[{\bf K}])= F(\bt')^2-F(\bt)^2,$
whence $F(\bt') > F(\bt).$ Therefore, $F$ is jointly increasing, and hence Lebesgue measurable (we include a proof in the Appendix; see Proposition~\ref{Lebmeas}). 
%By Lusin's theorem, given $\epsilon>0,$ for every measurable subset $A\subseteq (0,\infty)^{p}$ of finite measure there exists a closed set $E$ with $\mu(A\setminus E)< \epsilon$ such that $F$ restricted to $E$ is continuous.
Next, the set of discontinuities of any jointly non-decreasing function (which subsumes our case) has zero Lebesgue measure -- see Theorem~\ref{Discts} in the Appendix. In particular, the set of continuity points of $F$ is nonempty. 

Let ${\bf a} = (a_{1},\ldots, a_{p})\in (0,\infty)^{p}$ be a point of continuity of $F.$ To establish monotonicity of $F$ in each variable separately, fix $j\in [p]$ and let $\bt = (t_{1}, t_{2},\ldots, t_j, \ldots, t_{p}), \bt' = (t_{1}, t_{2},\ldots, t'_j \ldots, t_{p}) \in (0,\infty)^p$, with $t_j < t'_j$. Now consider the matrices
    \begin{equation*}
    A_j =
     \begin{pmatrix}
            t'_j  &  a_j \\
            t_j  &  a_j
        \end{pmatrix} \in \TP_{2 \times 2},
        \quad \text{and} \quad 
        A_i =
         \begin{pmatrix}
            t_i  &  a_i - \epsilon\\
            t_i  &  a_i
        \end{pmatrix},\
        \forall i \neq  j.
    \end{equation*}
    Note that all $A_i$ are $\TP_{2\times 2}$ if $0< \epsilon < \min\{a_i: i \neq j\}.$ Apply $F[-]$ entrywise to get
    \begin{equation*}
         F[A_{1}, A_{2}, \ldots, A_{p}] =
     \begin{pmatrix}
            F(t_{1}, t_{2},\ldots, t'_j, \ldots, t_{p})  &  F(a_{1}-\epsilon, a_{2}-\epsilon,\ldots, a_j, \ldots, a_{p} - \epsilon) \\
            F(t_{1}, t_{2},\ldots, t_j, \ldots, t_{p})  &  F(a_{1}, a_{2}, \ldots, a_j,\ldots, a_{p}).
        \end{pmatrix}
    \end{equation*}
    Then $\det F[A_{1},\ldots,A_{p}]>0$; now letting $\epsilon\to 0^{+}$, we obtain:
    \[
    F(t_{1}, t_{2},\ldots, t'_j, \ldots, t_{p}) \geq F(t_{1}, t_{2},\ldots, t_j \ldots, t_{p}).
    \]

    Thus, for each $1\leq j \leq n,$ the function $g_j : (0, \infty) \to (0, \infty)$ defined by 
    \[
   g_j(t) = F(t_{1},\ldots, t_{j-1}, t , t_{j+1}, \ldots, t_{p}) 
    \]
    is non-decreasing for all $t_{1},\ldots, \widehat{t_j}, \ldots t_{p} >0$, as claimed. In particular,
    $g_j(t - \epsilon) \leq g_j(t) \leq g_j(t + \epsilon)$ for each $0<\epsilon<t.$ Since $g_j$ is non-decreasing, its left and right hand limits exist, and we get
    \begin{equation}\label{nondec}
       \lim_{\epsilon \to 0^{+}} g_j(t -\epsilon) \leq g_j(t) \leq \lim_{\epsilon \to 0^{+}} g_j(t + \epsilon).
    \end{equation}
    
    To prove the continuity of $F$, consider the following $\TP_{2\times 2}$ matrices (with $0< \epsilon < t)$:
    \begin{equation*}
        A_j =
     \begin{pmatrix}
            t - \epsilon  &  t + \epsilon \\
            a_j  &  a_j+ (\frac{2a_j + 1}{t -\epsilon}) \epsilon
        \end{pmatrix}
        \quad \text{and} \quad 
        A_i =
         \begin{pmatrix}
            t_i  &  t_i \\
            a_i  &  a_i + \epsilon
        \end{pmatrix},\
        \forall i\neq j.
    \end{equation*}
    Apply $F[-]$ to get
    \begin{align*}
        &\ F(t_{1},\ldots, t -\epsilon, \ldots, t_{p}) F\bigg(a_{1}+\epsilon, \ldots, a_j+ {\textstyle\left(\frac{2a_j+1}{t-\epsilon}\right)} \epsilon, \ldots, a_{p}+ \epsilon\bigg)\\
        > &\ F(t_{1},\ldots, t + \epsilon, \ldots, t_{p}) F(a_{1},\ldots, a_j,\ldots, a_{p}).
    \end{align*}
    Taking limits,
    \begin{equation*}
        \lim_{\epsilon \to 0^{+}} F(t_{1},\ldots,t-\epsilon, \ldots, t_{p}) \geq \lim_{\epsilon \to 0^{+}} F(t_{1},\ldots, t +\epsilon, \ldots, t_{p}).
    \end{equation*}
    Combining this with \eqref{nondec},
       \begin{equation*}
       \lim_{\epsilon \to 0^{+}} g_j(t - \epsilon) = g_j(t) = \lim_{\epsilon \to 0^{+}} g_j(t + \epsilon), \quad \forall 1\leq j\leq p, \ t>0.
    \end{equation*}
    Thus, $g_j : (0, \infty) \to (0, \infty)$ is continuous for $j \in [p]$. By \cite{Young-cts} and \cite[Proposition 2]{KD-cts}, so is $F$.
\end{proof}

As promised, we now point out when a $\TP$ kernel can exist on $X \times Y$:

\begin{lemma}[{\cite[Lemma~4.5]{BGKP-TN}}]\label{Ltptosets}
The following are equivalent for nonempty totally ordered sets $X,Y$.
\begin{enumerate}
\item There exists a totally positive kernel $K : X \times Y \to \mathbb{R}$.

\item There exists a $\TP_2$ kernel $K : X \times Y \to \mathbb{R}$.

\item Either $X$ or $Y$ is a singleton, or there exist order-preserving
injections $: X,Y \hookrightarrow ( 0, \infty )$.
\end{enumerate}
The same equivalence holds if $Y = X$ and the kernels in $(1)$ and
$(2)$ are taken to be symmetric.
\end{lemma}

Given this result, it is possible to employ various $\TP$ test kernels that are restrictions of kernels on $(0,\infty)^2$ or even $\mathbb{R}^2$. With these tools at hand, we now show:

\begin{proof}[Proof of Theorem~\ref{TP-reduction}]
    Since $\mathcal{N}\geq 2,$ there exists a $\TP^{(2)}$ kernel on $X\times Y.$ Thus, using Lemma~\ref{Ltptosets} $X, Y$ can be identified -- via order-preserving bijections -- with subsets of $(0,\infty).$ 
    \par
    Moreover, as is explained in paragraph 2 of the proof of \cite[Theorem~4.6]{BGKP-TN}, every $2\times 2 \ \TP$ matrix is a rescaled generalized Vandermonde matrix
    \[
    \lambda^{-1}(u_{i}^{\alpha_{j}})_{i,j=1}^{2} \qquad \lambda, u_{1}, u_{2}>0,\ \alpha_{1}, \alpha_{2}\in \mathbb{R},
    \]
    where either $u_{1}<u_{2}$ and $\alpha_1 < \alpha_2$, or $u_{1}> u_{2}$ and $\alpha_{1}> \alpha_{2}.$ In particular, by Lemma~\ref{Ltptosets} and \cite[Theorem~4.6]{BGKP-TN}, every $2\times 2\ \TP$ matrix is (up to reversing the rows as well as the columns) a ``submatrix'' of the Vandermonde kernel $\lambda^{-1} e^{xy}$ on $X\times Y$ (which is $\TP$).

From this it follows that since $F[-] : \TP^{(k_1)}_{X \times Y} \times \cdots \times \TP^{(k_p)}_{X \times Y} \to \TP^{(l)}_{X \times Y}$, hence $F[-]$  sends $\TP_{2 \times 2} \times \cdots \times \TP_{2 \times 2} \to \TP_{2 \times 2}$. Thus, by the above Proposition~\ref{prop-TP continuity}, $F$ is continuous, positive, and coordinatewise non-decreasing, and hence admits a continuous extension to $[0,\infty)^p$, which we also denote by $F$. By Whitney's density theorem \cite{Whitney}, this extension $F[-]$ sends $(\TN_{2 \times 2})^p$ to $\TN_{2 \times 2}$. So by Theorem~\ref{Tpowers} and as $F$ is continuous and nonconstant, it is a mixed power function with $c > 0$ and ${\boldsymbol \alpha} \neq {\bf 0}$.
    
Finally, suppose that some $k_j = 1$ -- say for $j = 1.$ We recall that in the $\TN$ case, we had tools like constant kernels and padding by zeros available to us, but here in the $\TP$ case, we have neither of those; we remedy this by making use of the approximation technique and padding -- by \textit{ones} -- a suitable kernel. 

For all $j\neq 1,$ take the sequence of generalized Vandermonde kernels $\widetilde{K}^{(n)}_{j}(x,y) = e^\frac{xy}{n}$. Clearly, this is a sequence of $\TP$ kernels converging to the constant ($\TN$) kernel with all $1$s. As $\mathcal{N} \geq 2$, fix points $x_1 < x_2$ in $X$ and $y_1 < y_2$ in $Y$, and consider the test kernels evaluated at these pairs of points to be the matrices
\begin{equation}\label{Ekernel}
K_{1} = \begin{pmatrix} 1 & 2 \\ 2 & 1 \end{pmatrix}, \qquad \ 
K^{(n)}_{j} = \begin{pmatrix} e^{x_{1}y_{1}/n} & e^{x_{1}y_{2}/n} \\ e^{x_{2}y_{1}/n} & e^{x_{2}y_{2}/n} \end{pmatrix}, \forall j\neq 1.
\end{equation}

Let $\widetilde{K}_1 = \widetilde{K}^{(n)}_1 : X \times Y \to (0,\infty)$ be the kernel obtained via padding $K_{1}$ by $1$; note that $\widetilde{K}_{1} \in \TP_{X \times Y}^{(1)}
\setminus \TP_{X \times Y}^{(2)}$. Define $\widetilde{K}^{(n)}_{j}(x,y) = e^\frac{xy}{n} \in \TP_{X \times Y}^{(k_j)}$ for all other $j$, as above. By hypothesis, $F \circ {\bf
\widetilde{K}}^{(n)}$ is $\TP^{(2)}$ for all $n\geq 1$. As $F$ is continuous by Proposition~\ref{prop-TP continuity}, this implies
\[
0 \leq \lim_{ n \to \infty} \det (F \circ {\bf \widetilde{K}}^{(n)})[ \bx; \by ] =
c^2(1 - 4^{\alpha_1}),
\]
and so $\alpha_{1} = 0$, as desired.
\end{proof}

    Another significant part of the proof of Theorem~\ref{TmainTP} involves showing that if $\mathcal{N}\geq 3,$ the only multivariate $\TP$ transforms are individual powers (akin to the $\TN$ case). We next show this and also obtain additional information about which powers are permissible:

\begin{theorem}\label{TP-3}
   In the setting of Theorem~\ref{TmainTP} (and without the extra assumption in part~(4)), if $3\leq \mathcal{N}\leq \infty,$ there exists a unique $j_0 \in [p]$ such that $F(\bt)= ct_{j_0}^{\alpha_{j_0}},$ with
(i)~$c,\alpha_{j_0} > 0$ and
(ii)~$\alpha_{j_0} \in \mathbb{Z}_{>0} \cup (\mathcal{N}-2,\infty)$.
If moreover $\mathcal{N} = \infty$ then $\alpha_{j_0} = 1$.
\end{theorem}

Perhaps a curious coincidence to mention is that all uses of P\'olya frequency functions in this section appear in this one proof.

\begin{proof}
    Since $2\leq \mathcal{N}\leq \infty,$ Theorem~\ref{TP-reduction} gives that $F(\bt)= c \bt^{\boldsymbol{\alpha}}$ with $c > 0$ and ${\boldsymbol \alpha} \neq {\bf 0}.$ In particular, $F$ has a continuous extension to $[0,\infty)^p$, which we continue to denote by $F.$ (Here we retain the convention $0^{0} = 1$ whenever some $\alpha_{j}=0.$)
    
    We now temporarily ignore any mention or role of $\mathcal{N},$ and work out some preparation for this proof, including introducing the kernels that we will use. A function $\Lambda: \mathbb{R}\to \mathbb{R}$ is a \emph{P\'{o}lya frequency (PF) function} if $\Lambda$ is Lebesgue integrable, nonzero at $\geq 2$ points (which in fact implies on a semi-axis), and such that the associated Toeplitz kernel
    \begin{equation}\label{EToeplitzKernel}
    T_{\Lambda} : \mathbb{R}^{2} \to \mathbb{R}; \quad (x,y) \mapsto \Lambda(x-y)
    \end{equation}
    is totally nonnegative. These functions were introduced by Schoenberg in his landmark paper \cite{IJS-PF}, and a comprehensive account of the subject can be found in Karlin's monograph \cite{SK-TP}.

    It turns out that PF functions are discontinuous at most at one point (and if so, they vanish on an open semi-axis ending at this point). Moreover, all one-sided limits exist; see \cite{IJS-PF} for the proof. We say that a PF function $\Lambda$ is \emph{regular} if $\Lambda(a) = \frac{\Lambda(a^{+}) + \Lambda(a^{-})}{2},$ for all $a\in \mathbb{R}.$ We now mention the $\TN$ (not $\TP$, note) test kernels used in the proof below. Given scalars $a,\delta\in \mathbb{R}$ with $a>0,$ the map
    \begin{equation}\label{E1sidedPF}
        \Lambda_{a,\delta}(x) =
        \begin{cases}
            e^{-a(x-\delta)}, & \text{if}\ x > \delta, \\
            1/2, & \text{if}\ x = \delta, \\
            0, & \text{if}\ x < \delta
        \end{cases}
    \end{equation}
    is a regular P\'{o}lya frequency function -- it suffices to show this when $a=1, \delta = 0$, and in this case see \cite{AK} for the proof.

    With these preliminaries worked out, we return to the proof; now $3\leq \mathcal{N}\leq \infty.$ By the opening paragraph of this proof, $F(\bt)$ is a mixed power function $c  \bt^{\boldsymbol{\alpha}}.$ We claim there exists a unique $j_0 \in [p]$ such that $\alpha_{j_0} > 0.$ Suppose not -- say $\alpha_{1}, \alpha_{2} > 0.$ We make here a judicious choice of test kernels. Recall by Lemma~\ref{Ltptosets} that one can identify (via order-preserving bijections) $X, Y$ with subsets of $(0,\infty).$ Do so; then fix three points $\bx \in \inc{X}{3}$ and $\by \in \inc{Y}{3}.$ Now consider the order-preserving (injective) piecewise-linear maps $\varphi_{X}: X \to \mathbb{R}$ and $\varphi_{Y}: Y \to \mathbb{R}$ defined as follows:
    \[
\varphi_{X}(x) := 
\begin{cases}
    \frac{3}{2} + \frac{x-x_{2}}{x_{3}-x_{2}}, & \text{if}\ x\geq x_{2},\\
    \frac{3}{2} + \frac{x-x_{2}}{x_{2}-x_{1}}, & \text{if}\ x\leq x_{2},
\end{cases} \qquad
\varphi_{Y}(y) := 
\begin{cases}
     \frac{y-y_{2}}{y_{3} - y_{2}}, & \text{if}\ y\geq y_{2},\\
     \frac{y-y_{2}}{y_{2} - y_{1}}, & \text{if}\ y\leq y_{2}.
\end{cases}
    \]
    Thus, $\varphi_{X}: \bx \mapsto (\frac{1}{2}, \frac{3}{2}, \frac{5}{2})$ and $\varphi_{Y} : \by \mapsto (-1, 0 , 1).$ We now define the test kernels we will use: for each scalar $a>0$ and integer $1\leq j \leq p,$ define
    \[
\Lambda^{(a)}_{j}(x) :=
        \begin{cases}
            \Lambda_{a,0}(x), & \text{if}\ j = 1, \\
            \Lambda_{a,-3}(-x), & \text{if}\ j = 2, \\
            \Lambda_{a,-1}(x), & \text{if}\ j \geq 3.
        \end{cases}
    \]
    As mentioned in \cite{IJS-PF}, $\Lambda^{(a)}_{2}$ is (also) a regular PF function.

We can now complete the proof. As the PF functions $\Lambda^{(a)}_{j}$ are regular, by \cite[Proposition 8.9]{BGKP-TN} every $\Lambda^{(a)}_{j}$ can be approximated on $\mathbb{R}$ by a sequence of (strictly) totally positive P\'{o}lya frequency functions $\Lambda^{(a, n)}_{j} : \mathbb{R} \to (0,\infty).$ As $\varphi_{X}, \varphi_{Y}$ are order-preserving and injective, the kernels 
\[
K^{(a,n)}_{j} : X \times Y \to \mathbb{R}; \qquad (x,y)\mapsto T_{\Lambda^{(a,n)}_{j}}(\varphi_{X}(x), \varphi_{Y}(y)) = \Lambda^{(a,n)}_{j}(\varphi_{X}(x) - \varphi_{Y}(y))
\]
are also $\TP$ for all $j\in [p]$ and $n\geq 1.$ Thus by hypothesis, the sequence of kernels
\[
F \circ {\bf K}^{(a,n)} : X \times Y \to [0,\infty)
\]
are all $\TP^{(l)}.$ By the continuity on the closed orthant $[0,\infty)^{p}$ of (the extension of) $F,$ their pointwise limit is now $\TN^{(l)}$ on $X \times Y :$
\[
(x,y)\mapsto  F\bigg(\Lambda^{(a)}_{1}(\varphi_{X}(x) - \varphi_{Y}(y)),\dots, \Lambda^{(a)}_{p}(\varphi_{X}(x) - \varphi_{Y}(y))\bigg).
\]
As $l\geq \mathcal{N} \geq 3,$ the submatrix at $\bx,\by,$
\[
M^{(a)}_{3\times 3} := \bigg(F\bigg(\Lambda^{(a)}_{1}(\varphi_{X}(x_{i}) - \varphi_{Y}(y_{j})),\dots, \Lambda^{(a)}_{p}(\varphi_{X}(x_{i}) - \varphi_{Y}(y_{j}))\bigg)\bigg)_{i,j =1}^{3}
\]
has nonnegative determinant for each $a > 0.$ But $M^{(a)}$ is the Schur product of the Toeplitz matrices
\begin{align*}
    A_{1} &= c \begin{pmatrix}
e^{-3a\alpha_{1}/2} & e^{-a\alpha_{1}/2} & 0 \\ e^{-5a\alpha_{1}/2} & e^{-3a\alpha_{1}/2} & e^{-a\alpha_{1}/2} \\ e^{-7a\alpha_{1}/2} & e^{-5a\alpha_{1}/2} & e^{-3a\alpha_{1}/2} \end{pmatrix}, \\
A_{2} &=  \begin{pmatrix}
e^{-3a\alpha_{2}/2} & e^{-5a\alpha_{2}/2} & e^{-7a\alpha_{2}/2}\\  e^{-a\alpha_{2}/2} & e^{-3a\alpha_{2}/2} & e^{-5a\alpha_{2}/2}\\ 0 & e^{-a\alpha_{2}/2} & e^{-3a\alpha_{2}/2} \end{pmatrix}, \\
A_{j} &=  \begin{pmatrix}
e^{-5a\alpha_{j}/2} & e^{-3a\alpha_{j}/2} & e^{-a\alpha_{j}/2}\\  e^{-7a\alpha_{j}/2} & e^{-5a\alpha_{j}/2} & e^{-3a\alpha_{j}/2}\\ e^{-9a\alpha_{j}/2} & e^{-7a\alpha_{j}/2} & e^{-5a\alpha_{j}/2} \end{pmatrix}, \quad 3\leq j\leq p.
\end{align*}
In other words,
\[
M^{(a)} = c \begin{pmatrix}
e^{-a\kappa_{0}/2} & e^{-a\kappa_{-1}/2} & 0\\  e^{-a\kappa_{1}/2} & e^{-a\kappa_{0}/2} & e^{-a\kappa_{-1}/2}\\ 0 & e^{-a\kappa_{1}/2} & e^{-a\kappa_{0}/2} \end{pmatrix}
\]
where
\[
\kappa_{-1}= \alpha_{1} +5 \alpha_{2} + 3 \sum_{j=3}^{p} \alpha_{j}, \quad \kappa_{0}= 3\alpha_{1} + 3 \alpha_{2} + 5 \sum_{j=3}^{p} \alpha_{j}, \quad \kappa_{1} = 5 \alpha_{1} + \alpha_{2} + 7 \sum_{j=3}^{p} \alpha_{j}.
\]
are positive real scalars. Now since $\det M^{(a)}\geq 0$ for all $a> 0,$ it follows that 
\[
0 \leq \lim_{a\to 0^{+}} \det M^{(a)} = \det \lim_{a\to 0^{+}} M^{(a)} = c^3 \det \begin{pmatrix}
1 & 1 & 0 \\ 1 & 1 & 1 \\ 0 & 1 & 1 \end{pmatrix}.
\]
As this is false, so is the assumption that two $\alpha_{j} > 0.$\qed\medskip

It remains to show~(ii), and that $\alpha_{j_0} = 1$ if $\mathcal{N} = \infty$. We begin with a preliminary lemma.

\begin{lemma}\label{LregularTN}
Fix integers $q \in \mathbb{Z}_{>0}$ and $l \in \mathbb{Z}_{>0} \sqcup \{\infty\}$, and let $X,Y \subset (0,\infty)$ be such that $3 \leq \mathcal{N} = \min \{ |X|, |Y|, l \} \leq \infty$. If $F : [0,\infty)^q \to [0,\infty)$ is continuous and $F[-] : (\TP_{X \times Y})^q \to \TN_{X \times Y}^{(l)}$, then
$F[T_{\Lambda_1}, \dots, T_{\Lambda_q}] \in \TN^{(\mathcal{N})}_{\mathbb{R} \times \mathbb{R}}$  for all regular P\'olya frequency functions $\Lambda_1, \dots, \Lambda_q$.

A similar result holds for symmetric kernels: if $F[-] : (\STP_X)^q \to \STN_X^{(l)}$, and $|X| \geq 2l$, then $F[T_{\Lambda_1}, \dots, T_{\Lambda_q}] \in \STN^{(l)}_{\mathbb{R} \times \mathbb{R}}$  for all even P\'olya frequency functions $\Lambda_j$ (with $T_{\Lambda_j}$ as in~\eqref{EToeplitzKernel}).
\end{lemma}

\begin{proof}
Fix a finite integer $1 \leq r \leq \mathcal{N}$, and choose any points ${\bf a}, {\bf b} \in \inc{\mathbb{R}}{r}$. As $r \leq |X|$, we choose and fix an increasing $r$-tuple $\bx \in \inc{X}{r}$. Now define the order-preserving injection $\varphi_X : X \to \mathbb{R}$ given by:
$\varphi_X(x_i) := a_i$, and $\varphi_X$ is piecewise linear and strictly increasing on the intervals
\[
(-\infty, x_1) \cap X, \quad (x_1, x_2) \cap X, \quad \dots, \quad (x_{r-1}, x_r) \cap X, \quad (x_r, \infty) \cap X.
\]
Similarly choose $\by \in \inc{Y}{r}$ and define $\varphi_Y : Y \hookrightarrow \mathbb{R}$, sending each $y_i$ to $b_i$.

We now proceed. As above, we use \cite[Proposition 8.9]{BGKP-TN} to approximate each $\Lambda_j$ by a sequence of (strictly) totally positive PF functions $\Lambda_j^{(n)}$. Then define the kernels
\[
K_j^{(n)} : X \times Y \to \mathbb{R}, \qquad (x,y) \mapsto \Lambda_j^{(n)}(\varphi_X(x) - \varphi_Y(y)), \qquad j \in [q], \ n \geq 1.
\]

Thus $K_j^{(n)} \in \TP_{X \times Y}$ for all $j,n$, so the hypotheses yield that
\[
F[K_1^{(n)}, \dots,K_q^{(n)}][\bx; \by] = F[T_{\Lambda_1^{(n)}}, \dots, T_{\Lambda_q^{(n)}}][{\bf a}; {\bf b}] \in \TN_{r \times r}.
\]
As $F$ is continuous, letting $n \to \infty$ yields
$F[ T_{\Lambda_1}, \dots, T_{\Lambda_q}][{\bf a}; {\bf b}] \in \TN_{r \times r}$; as this holds for all ${\bf a}, {\bf b} \in \inc{\mathbb{R}}{r}$ and all $1 \leq r \leq \mathcal{N}$, the result follows.

The proof is similar in the ``symmetric'' case: as P\'olya frequency functions do not have compact support, each even PF function is nowhere vanishing and hence continuous. As in the proof of \cite[Proposition~8.9]{BGKP-TN} we approximate each $\Lambda_j$ by its convolution with the Gaussian density with variance $1/n$; each such is a totally positive even PF function, say $\Lambda_j^{(n)}$. Now let ${\bf a}, {\bf b} \in \inc{\mathbb{R}}{r}$ for a finite integer $1 \leq r \leq l$, and merge their coordinates in (strictly) increasing order into a real tuple ${\bf c} \in \mathbb{R}^m$, say. As $m \leq 2l \leq |X|$, we now choose points $x_1 < \cdots < x_m$ in $X$, and an order-preserving, piecewise linear injection $\varphi_X : X \to \mathbb{R}$ sending $x_i \mapsto c_i$ for all $i \in [m]$ as above. Now let $\bx := \varphi_X^{-1}({\bf a}), \by := \varphi_X^{-1}({\bf b}) \in \inc{X}{r}$.
Also define the kernels
\[
K_j^{(n)} : X \times X \to \mathbb{R}, \qquad (x,x') \mapsto \Lambda_j^{(n)}(\varphi_X(x) - \varphi_X(x')), \qquad j \in [q], \ n \geq 1.
\]
Then $K_j^{(n)} \in \STP_X$ for all $j,n$, so working as above, the result follows.
\end{proof}

We now return to the proof of~(ii), using the kernel
\begin{equation}
\Omega(x) := (\Lambda_{1,0} \ast \Lambda_{1,0})(x) = x e^{-x} {\bf 1}_{x>0}
\end{equation}
(see~\eqref{E1sidedPF}, and note from~\cite{IJS-PF} that PF functions are closed under convolution). The powers of this kernel were studied by Karlin in 1964~\cite{KarlinTAMS}, and he showed that for each integer $r \geq 2$, if $\alpha \in \mathbb{Z}_{\geq 0} \cup (r-2,\infty)$ then $\Omega(x)^\alpha$ is $\TN^{(r)}$. The converse was recently shown by one of us~\cite{AK}: if $\alpha \in (0,r-2) \setminus \mathbb{Z}$ then $\Omega(x)^\alpha$ is not $\TN^{(r)}$.

Now we put all of this together. From the opening paragraph of this proof, every preserver $F$ satisfying~\eqref{ETPq} extends to $[0,\infty)^p$ and equals $F({\bf t}) = c t_{j_0}^{\alpha_{j_0}}$ for some $j_0 \in [p]$ and $c, \alpha_{j_0} > 0$. Thus $F$ satisfies the hypotheses of Lemma~\ref{LregularTN} with $q=1$, so $F[\Omega] \in \TN^{(\mathcal{N})}_{\mathbb{R} \times \mathbb{R}}$. By the preceding paragraph, the positive power $\alpha_{j_0} \in \mathbb{Z}_{>0} \cup (\mathcal{N}-2, \infty)$.

Finally, if $\mathcal{N} = \infty$, we employ \cite[Lemma~8.4]{BGKP-TN}, which says that
\begin{equation}\label{EevenPF}
M_\gamma(x) := (\gamma+1) \exp(-\gamma |x|) - \gamma \exp(-(\gamma+1)|x|), \qquad \gamma > 0
\end{equation}
is a family of (regular) PF functions such that $M_\gamma(x)^n$ is not $\TN$ for any integer $n>1$ (and any real $\gamma > 0$). Again applying Lemma~\ref{LregularTN} with $q=1$, $F[M_\gamma] = M_\gamma^{\alpha_{j_0}}$ being $\TN$ implies that $\alpha_{j_0} = 1$.
\end{proof}

Theorems~\ref{TP-reduction} and~\ref{TP-3} show a majority of the assertions in our main result, Theorem~\ref{TmainTP} -- specifically, in parts~(2) and~(3), respectively. To show the final part~(4) requires a final preliminary result, which is now stated for general $q \geq 1$ but used below for $q=1$, and which isolates the use of Theorem~\ref{Twhitney2} in proving Theorem~\ref{TmainTP}(4).

\begin{proposition}\label{Pinflate}
Fix integers $q \in \mathbb{Z}_{>0}$ and $k_1, \dots, k_q \in \mathbb{Z}_{>0}$, and $l \in \mathbb{Z}_{>0} \sqcup \{\infty\}$. Also fix subsets $X,Y \subseteq \mathbb{R}$ of size at least 2, possibly infinite. Now suppose a function $F : [0,\infty)^q \to [0,\infty)$ is continuous, and such that the entrywise map $F[-] : \TP_{X \times Y}^{(k_1)} \times \cdots \times \TP_{X \times Y}^{(k_q)} \to \TN_{X \times Y}^{(l)}$.
Then for all $1 \leq m \leq |X|$ and $1 \leq n \leq |Y|$, and matrices $A_j \in \TN_{m \times n}^{(k_j)}$, the matrix $F[ A_1, \dots, A_q ] \in \TN_{m \times n}^{(l)}$.

A similar result holds for symmetric kernels: if $F[-] : \STP_X^{(k_1)} \times \cdots \times \STP_X^{(k_q)} \to \STN_X^{(l)}$ and $A_j \in \STN_{m \times m}^{(k_j)}$ for all $j$ (with $1 \leq m \leq |X|$), then $F[ A_1, \dots, A_q ] \in \STN_{m \times m}^{(l)}$.
\end{proposition}

\begin{proof}
Choose and fix points $\bx \in \inc{X}{m}$ and $\by \in \inc{Y}{n}$, and choose any $\epsilon > 0$ such that
\[
\epsilon < \frac{1}{2}\min\{ x_{i + 1} - x_i, \ y_{j + 1} - y_j : i \in [m-1], \ j \in [n-1] \}.
\]
Now define the ``inflation'' to $X \times Y$ of a real $m \times n$ matrix $A$ to be the kernel
\[
K_A = K_{A,\epsilon} : X \times Y \to \mathbb{R}; \qquad ( x, y ) \mapsto %
\begin{cases}
a_{i j}, & \textrm{if } | x - x_i | < \epsilon \text{ and } | y - y_j | < \epsilon \quad (i\in [m], \ j \in [n]), \\
0, & \textrm{otherwise}.
\end{cases}
\]

Then each $K_{A_j}$ is a bounded $\TN^{(k_j)}$ kernel on $X \times Y$, since any square submatrix of size at most $k_j \times k_j$ either has a row or column of zeros, or two equal rows or columns, or is a square submatrix of $A_j$. Hence by Theorem~\ref{Twhitney2}, there exists a sequence of $\TP^{(k_j)}$ kernels $( K_{A_j}^{(N)} )_{n \geq 1}$ converging to $K_{A_j}$ at -- points of continuity -- $( x_i, y_j )$ for all $i,j$. (If $X,Y$ are finite then one can bypass this argument and appeal to Whitney density \cite{Whitney} directly.) The hypotheses yield that $F \circ {\bf K}^{(N)}_{\bf A} \in \TN_{X \times Y}^{(l)}$, whence $(F \circ {\bf K}^{(N)}_{\bf A})[ \bx; \by] \in \TN_{m \times n}^{(l)}$ for all $N$. As $F$ is continuous, taking $N \to \infty$ yields the result.
The proof of the symmetric counterpart is similar.
\end{proof}

With the above results at hand, we prove our main result.

 \begin{proof}[Proof of Theorem~\ref{TmainTP}]
     \begin{enumerate}
         \item Suppose $\mathcal{N}=1.$ Since we can only take minors of $F\circ \mathbf{K}$ that are of size $\mathcal{N}\times \mathcal{N}$ or smaller, $F$ can be any function taking positive values on the positive orthant.
         
         \item Henceforth, $\mathcal{N}\geq 2$. By Lemma~\ref{Ltptosets}, $X, Y$ can be embedded (via order-preserving bijections) into subsets of $(0,\infty)$. Moreover, by Theorem~\ref{TP-reduction}, every preserver is of the form $F(\bt) = c \bt^{\boldsymbol{\alpha}}$ with all $\alpha_j \geq 0$ and $\boldsymbol{\alpha} \neq {\bf 0}$. In particular, $F$ admits a continuous extension to $[0,\infty)^p$, which we will continue to denote by $F$.
         
         Now say $\mathcal{N}=2$. Then one implication follows from Theorem~\ref{TP-reduction}.
Conversely (for $\mathcal{N} = 2$), suppose $F$ is of the given
form, and $K_j \in \TP_{X \times Y}^{(k_j)}$ for all $j$. As $F$ is positive-valued,
$F \circ {\bf K}$ is $\TP^{(1)}$. Now fix $\bx \in \inc{X}{2}$ and $\by
\in \inc{Y}{2}$; if $k_j \geq 2$ then $K_{j}(x_1,y_1) K_{j}(x_2,y_2) >
K_{j}(x_1,y_2) K_{j}(x_2,y_1) > 0$. Raising both sides to the $\alpha_j$th power, and
multiplying over all $j$ with $k_j \geq 2$ (and also by $c > 0$), we
see that $\det (F \circ {\bf K})[ \bx; \by ] > 0$. Note that the case when all $k_{j} = 1$ is not admissible since if we take the test kernels $K_{j} = \widetilde{K}_{1}$ (see~\eqref{Ekernel} and thereafter) for all $j,$ then the above computation shows 
\[
0 < \det (F \circ {\bf K})[ \bx; \by ] =
c^p(1 - 4^{\alpha_{1} + \cdots + \alpha_{p}}),
\]
which is absurd since all $\alpha_j \geq 0$ and ${\boldsymbol \alpha} \neq {\bf 0}$. This concludes the proof for $\mathcal{N} = 2.$

     \item Next, suppose $\mathcal{N} = 3$. By Theorem~\ref{TP-3} there exists a unique $j_0 \in [p]$ such that $F(\bt) = c t_{j_0}^{\alpha_{j_0}}$ for some $c > 0$ and $\alpha_{j_0} \geq 1$. We now show that $k_{j_0} \geq \mathcal{N} = 3$. To do so, we employ the Jain--Karlin--Schoenberg kernel $K_{\mathcal{JKS}}: \mathbb{R} \times \mathbb{R} \to \mathbb{R}$, defined in \cite[Definition 1.13]{AK} by
\[
K_{\mathcal{JKS}}(x,y) := \max(1+xy, 0).
\]
Here, we consider the restriction of $K_{\mathcal{JKS}}$ to the domain $(0,\infty)^{2}$, where it equals $1+xy$. Now say $k_{j_0} \leq 2$. Fix $0 < \epsilon < 1$ and set
$K_{j_0} := K_{\mathcal{JKS}}^{{\alpha_{j_0}^{-1}}\epsilon}$,
and all other test kernels (which anyway do not get used) to equal any $\TP$ kernel, say $(x,y) \mapsto e^{xy}$. Then $K_{j_0} \in \TP^{(k_{j_0})}_{X \times Y}$ by \cite[Theorem~C]{AK}, so $F\circ {\bf K} \in \TP^{(\mathcal{N})}_{X\times Y}$ by the hypothesis. That is,
$K_{j_0}^{\alpha_{j_0}} = K_{\mathcal{JKS}}^\epsilon$ is in $\TP^{(3)}_{X\times Y}$. But this is false by \cite[Theorem~C]{AK}, since $\epsilon \in (0, 1)$. So $k_{j_0} \geq \mathcal{N}$.

Conversely, say $F(\bt) = c t_{j_0}^{\alpha_{j_0}}$ is of the desired
form, and $K_{j_0} \in \TP_{X \times Y}^{(k_{j_0})}$. Since we are only
considering $3 \times 3$ minors, it suffices to show that $\TP_{3 \times
3}$ is closed under taking entrywise $\alpha$th powers for $\alpha \geq
1$. But this is precisely \cite[Theorem 4.4]{FJS}.

\item Finally, let $4 \leq \mathcal{N}\leq \infty.$ Again by Theorem~\ref{TP-3} we get that $F(\bt) = c t_{j_0}^{\alpha_{j_0}}$ for a unique $j_0 \in [p]$. First suppose $\mathcal{N} = \infty$. Then Theorem~\ref{TP-3} implies that $\alpha_{j_0} = 1$, and it remains to show that $k_{j_0} = \infty$. (The converse, that if $k_{j_0} = \infty$ then each such homothety preserves $\TP_{X \times Y}^{(\infty)}$, is obvious.)
Thus, suppose to the contrary that $1 \leq k_{j_0} < \infty = \mathcal{N}$. Choose any $\delta \in (k_{j_0}-1, k_{j_0})$, and let the kernel $K_{j_0} := K_{\mathcal{JKS}}^{\delta}$. By \cite[Theorem~C]{AK}, $K_{j_0} \in \TP_{X \times Y}^{(k_{j_0})} \setminus \TP_{X \times Y}^{(\infty)}$. This contradicts that $c K_{j_0} \in \TP^{(\infty)}_{X \times Y}$.

The final sub-case is when $4 \leq \mathcal{N} < \infty$. As above, Theorem~\ref{TP-3} implies $F({\bf t}) = c t_{j_0}^{\alpha_{j_0}}$, with $\alpha_{j_0} \in \mathbb{Z}_{>0} \cup (\mathcal{N}-2,\infty)$.
We now invoke the hypotheses that all $k_{j_0} < \infty$ if $X,Y$ are infinite. Now the goal is to apply Proposition~\ref{Pinflate} with $q=1$ for (the extension to $[0,\infty)$ of) $t^{\alpha_{j_0}}$; for this, we need to ensure that the hypotheses of Proposition~\ref{Pinflate} follow from those of Theorem~\ref{TmainTP}(4). Set $k'_{j_0} := \min \{ k_{j_0}, |X|, |Y| \} < \infty$, and apply Proposition~\ref{Pinflate} with
\[
q=1, \quad k_{j_0} \leadsto k'_{j_0}, \quad m=n=4, \quad F(t) = c t^{\alpha_{j_0}},
\]
and $A$ the $\TP_{4\times 4}$ matrix in \cite[Example 5.8]{FJS} (for suitable $x,\epsilon > 0$). This yields $\alpha_{j_0} = 1$.

Finally, we claim that $k_{j_0} \geq \mathcal{N} \geq 4$. This concludes the proof, since as above, the converse assertion (that if $k_{j_0} \geq \mathcal{N}$ then homotheties $c t_{j_0}$ satisfy~\eqref{ETPq}) is immediate. To show the claim, suppose for contradiction that $1 \leq k_{j_0} < \mathcal{N} < \infty$. Again use the Jain--Karlin--Schoenberg kernel: choose $\delta \in (\mathcal{N}-3, \mathcal{N} - 2)$, and let $K_{j_0} := K_{\mathcal{JKS}}^{\delta}$ on $X \times Y \subseteq (0,\infty)^2$. By \cite[Theorem~C]{AK},
\[
K_{j_0} \in \TP_{X \times Y}^{(\mathcal{N}-1)} \setminus \TP_{X \times Y}^{(\mathcal{N})}
\subseteq \TP_{X \times Y}^{(k_{j_0})} \setminus \TP_{X \times Y}^{(\mathcal{N})}.
\]
This contradicts that $c K_{j_0} \in \TP^{(l)}_{X \times Y}$, since $l \geq \mathcal{N}$.\qedhere
\end{enumerate}
 \end{proof}

\section{Transforms of symmetric $\TN$ kernels}

We next turn to the symmetric counterparts, starting with Theorem~\ref{TmainSTN}. A crucial part of the proof involves showing continuity:

\begin{proposition}\label{Pmidconvex}
Given an integer $p \geq 1$ and a function $F : (0,\infty)^p \to [0,\infty)$, if $F[-] : (\STP_{2 \times 2})^p \to \STN_{2 \times 2}$, then $F$ is nonnegative, jointly non-decreasing, multiplicatively mid-convex, and continuous.
\end{proposition}

Notice that this result is stronger than Proposition~\ref{prop-TP continuity}, because its hypotheses are weaker. That said, we have retained that proof (in the Appendix) because it is self-contained, whereas the proof here uses a nontrivial classical result -- which has an involved proof -- to show continuity. Moreover, the two proofs are necessarily distinct because the earlier one uses judiciously chosen asymmetric test matrices, and the proof below necessarily uses only symmetric ones.

\begin{proof}
That $F$ is nonnegative immediately follows from $F[-]$ being $\TN_{2 \times 2}$. Next say $0 < t_j < t'_j$ for all $j \in [p]$. Let $K_j := \begin{pmatrix} t_j' & t_j \\ t_j & t'_j \end{pmatrix}$ for all $j$; then $F[ {\bf K}] \in \TN_{2 \times 2}$. Thus
$0 \leq \det(F[{\bf K}])= F(\bt')^2-F(\bt)^2,$
whence $F(\bt') \geq F(\bt).$

The proof of multiplicative mid-convexity is in two steps, which are separated by the proof of continuity. Namely, first given ${\bf t} \in (0,\infty)^p$, define
\[
g_{\bf t} : [0, \infty) \to \mathbb{R}; \qquad \epsilon \mapsto F(\bt + (\epsilon,\dots,\epsilon)).
\]
As $g_{\bf t}$ is non-decreasing (from above) and nonnegative, one can now define the function
\[
F^\nearrow : (0,\infty)^p \to [0,\infty); \qquad \bt \mapsto \lim_{\epsilon \to 0^+} F(\bt + (\epsilon,\dots,\epsilon)).
\]

The plan for the proof is now as follows:
(a)~we first claim that $F^\nearrow$ is multiplicatively mid-convex;
(b)~we next use this to show that $F^\nearrow$ is continuous along the main diagonals $(1,\dots,1)$;
(c)~this is used to show the continuity of $F$ on the entire orthant;
(d)~and already from~(b) we see that $F = F^\nearrow$ is multiplicatively mid-convex.

We begin with~(a). Let $\bt, \bt' \in (0,\infty)^p$, and consider the test matrices
\[
K_j^{(\epsilon)} := \begin{pmatrix} t_j + \epsilon & \sqrt{t_j t'_j} + \epsilon \\ \sqrt{t_j t'_j} + \epsilon & t'_j + \epsilon \end{pmatrix} \in \STP_{2 \times 2}, \quad \epsilon > 0.
\]
Then $F[ {\bf K}^{(\epsilon)} ] \in \STN_{2 \times 2}$, so
\[
0 \leq \lim_{\epsilon \to 0^+} \det F[ {\bf K}^{(\epsilon)} ] = F^\nearrow(\bt) F^\nearrow(\bt') -  F^\nearrow \left( \sqrt{t_1 t'_1}, \dots, \sqrt{t_p t'_p} \right)^2.
\]
Thus $F^\nearrow$ is multiplicatively mid-convex. Now fix $\bt$ and restrict $F$ to the ``diagonal semi-axis''
\[
\mathcal{L}_\bt =
\{ \bt(\epsilon) := \bt + (\epsilon, \dots, \epsilon) : - \| \bt \|_\infty < \epsilon < \infty \}.
\]
On this set, $F^\nearrow(\bt(\epsilon))$ is non-decreasing in $\epsilon$ and hence a measurable function of one real variable. It follows either from \cite[Theorem 71.C]{RV} (which is modeled after a result by Ostrowski \cite{Ostr}), or else by a result originally due to Blumberg \cite{Blumberg} and Sierpi\'nsky \cite{Sierpinsky}, that $F^\nearrow$ is continuous on $\mathcal{L}_\bt$ for all $\bt \in (0,\infty)^p$. This shows~(b).

Next, note that along the diagonals $\mathcal{L}_\bt$, $F$ is non-decreasing and hence has only countably many jump discontinuities. Hence so does $F^\nearrow$; but $F^\nearrow$ is continuous on $\mathcal{L}_\bt$, so $F$ is as well -- and $F^\nearrow \equiv F$ on $\mathcal{L}_\bt$. This also shows~(d).

Finally, to show~(c) we fix $\bt_0 \in (0,\infty)^p$ and show that $F$ is continuous at $\bt_0$. Given $\epsilon > 0$, there exists $\delta \in (0, \| \bt_0 \|_\infty)$ such that if  $|s| < \delta$ then $F(\bt_0 + (s,\dots,s)) = F^\nearrow(\bt_0 + (s,\dots,s))$ takes values in $(F(\bt_0) - \epsilon, F(\bt_0) + \epsilon)$. Now the open ball of radius $\delta$ is inscribed in the open cube with the same center and edge length $2 \delta$; and $F$ is jointly non-decreasing. These imply that if $\| \bt - \bt_0 \| < \delta$ then $|F(\bt) - F(\bt_0)| < \epsilon$, proving that $F$ is continuous on $(0,\infty)^p$.
\end{proof}

We next prove a $\STN$ completion result that is useful below.

\begin{proposition}\label{TNcompletion}
    Every $\TN_{2 \times 2}$ matrix can be embedded into a $\STN_{3 \times 3}$ matrix.
\end{proposition}

\begin{proof}
Let $A = \begin{pmatrix}
        a & b \\ c & d 
    \end{pmatrix}$ 
    be any $\TN$ matrix. First if $b=c=0$ then $A \in \STN_{2 \times 2}$, hence can be padded by zeros to land in $\STN_{3 \times 3}$. Next if say $c \neq 0$ (so $c>0$), we embed $A$ into the matrix
      \begin{equation}\label{ETNcompletion}
        \Tilde{A} = \begin{pmatrix}
       \ast  & a & b \\ a & c & d \\ b & d & \frac{d^2 +1}{c}.
    \end{pmatrix}
      \end{equation}
  Note that $\Tilde{A}$ is partial $\TN$, i.e., all minors not including $\ast$ are already nonnegative.
  Thus, $\Tilde{A} \in \STN_{3 \times 3}$ if and only if $\ast$ satisfies the inequalities
  (a)~$d \cdot \ast \geq ab$ (which is feasible if $d>0$; and is trivial if $d=0$, since then $b=0$ as $A$ is $\TN$);
  (b)~$\ast \geq 0$;
  (c)~$\ast \geq a^2/c$;
  (d)~$\ast \geq b^2c/(d^2+1)$;
  (e)~$\det \Tilde{A} \geq 0$, i.e.\ $\ast \geq -2abd + b^2c + a^2(d^2+1)/c$.
As these are simultaneously feasible, such an $\Tilde{A} \in \STN_{3 \times 3}$ exists.

The final case is if $b \neq 0 = c$. In this case, we embed $B := A^T$ into $\Tilde{B}$, and then complete it as above. Then the bottom left $2 \times 2$ submatrix of $\Tilde{B} \in \STN_{3 \times 3}$ is $A$.
\end{proof}

With Propositions~\ref{Pmidconvex} and \ref{TNcompletion} at hand, we proceed to the main result.

\begin{proof}[Proof of Theorem~\ref{TmainSTN}]
 The result is obvious for $\mathcal{N}=1.$
 
    \begin{enumerate}
    \item Before working out the $\mathcal{N}=2$ case, we first develop some results that will apply in all parts below. Thus, henceforth we assume $2 \leq \mathcal{N} \leq \infty$. First suppose $F[-]$ satisfies~\eqref{ESTPq}. We now apply Proposition~\ref{Pmidconvex} (even with the weaker assumptions) to get that $F$ is nonnegative, jointly non-decreasing, multiplicatively mid-convex, and hence continuous on $(0,\infty)^p$.

    Now suppose moreover that $3 \leq |X| \leq \infty$; choose and fix $\bx \in \inc{X}{3}$. We show that up to rescaling, $F$ is a product of power functions and Heaviside indicators as in Theorem~\ref{Tpowers}. The idea is to repeat the steps in that proof, with the necessary modifications via embedding the test $\TN_{2 \times 2}$ matrices (padded by zeros) within $\STN_{3 \times 3}$ matrices (padded by zeros). Thus, henceforth we work exclusively with matrices in $\STN_{3 \times 3}$ (arising from the kernels evaluated at $[\bx; \bx]$), and avoid mentioning the padding by zeros to $X \times X$.

We begin with Step~1 of the proof of Theorem~\ref{Tpowers} -- i.e., Lemma~\ref{Lclaim}. Fix an index $j \in [p]$ and scalars $t_1, \dots, \widehat{t_j}, \dots, t_p > 0$. Let the $\STN$ kernels $K_i[\bx; \bx] = t_i {\bf 1}_{3 \times 3}$ for $i \neq j$, and let $K_j[\bx; \bx]$ run over $\STN_{3 \times 3}$. Then $g_j(-) := F(t_1,\dots, -, \dots, t_p)$ entrywise maps $\STN_{3\times 3}\to \STN^{(2)}_{3\times 3}$ (since $|X|\geq 3$), by Lemma~\ref{Lpadding}. By Proposition~\ref{TNcompletion}, $g_j[-] : \TN_{2\times 2}\to \TN_{2\times 2}$, and \cite[Theorem 3.3]{BGKP-TN} now gives up to rescaling that $g_j(-)$ equals a power or Heaviside indicator. In particular, $F$ satisfies the hypotheses of Lemma~\ref{Lclaim} on $(0,\infty)^p$, and so $F(\bt) = c \bt^{\boldsymbol \alpha}$ for $\bt > {\bf 0}$, with $c > 0$ and $\alpha_j \geq 0\ \forall j \in [p]$.

Next, we suitably repeat Step~2. Namely, as all matrices in~\eqref{Estep21} and~\eqref{Estep22} are symmetric and $\TN$, one can attach a zero row and column at the end of each (and then pad by zeros to $X \times X$). Now the determinantal inequalities deduced there imply that the conclusions of Step~2 hold in the present situation as well.

We next come to Step~3 part~(1). Here, all test matrices in~\eqref{Estep31} have $(2,1)$ entry $1$ -- hence by the discussion around~\eqref{ETNcompletion}, they all embed in the top right $2 \times 2$ ``corner'' of matrices $\widetilde{A}_j \in \STP_{3 \times 3}^{(2)}$. Now the non-negativity of the minor $F[ {\bf A} ][(1,2); (2,3)]$ yields the conclusions of Step~3 part~(1).

Finally, we come to Step~3 part~(2). Thus, $t_{j_0} > 0$ for all $j_0 \in J_{00}$, and we need to show~\eqref{Eclaim1}, where $J_0 \setminus J_{00} = [p_0]$ and $[p] = [p_0] \sqcup J_{00} \sqcup J_+$. We proceed by induction on $l \in [0,p_0]$. The base case of $l=0$ is true via Step~1; now let the rank one $\STN$ matrices
\begin{equation}
A_1 := \begin{pmatrix} 1 & 0 & 1 \\ 0 & 0 & 0 \\ 1 & 0 & 1 \end{pmatrix}, \qquad
A_2 = \cdots = A_l := \begin{pmatrix} 1 & 1 & 0 \\ 1 & 1 & 0 \\ 0 & 0 & 0 \end{pmatrix}, \qquad
A_j := \begin{pmatrix} 1 & 1 & t_j \\ 1 & 1 & t_j \\ t_j & t_j & t_j^2 \end{pmatrix}, \ j>l.
\end{equation}
Thus $F[{\bf A}] \in \STN_{3 \times 3}^{(2)}$ by the hypotheses, so by the induction hypothesis for $l-1$ we get:
\begin{align*}
0 \leq \det F[ {\bf A} ][(1,2); (1,3)] = &\ F({\bf 1}_p) F({\bf 0}_l, \bt'') - F(1, {\bf 0}_{l-1}, \bt'') F(0, {\bf 1}_{p-1})\\
= &\ c \left( F({\bf 0}_l, \bt'') - c \prod_{j \in J_+} t_j^{\alpha_j} \right),
\end{align*}
where $F(0, {\bf 1}_{p-1}) = c$ by definition of $J_0, J_{00}$ (and not by induction). Similarly,
\begin{align*}
0 \leq \det F[ {\bf A} ][(2,3); (2,3)] = &\ F(0, {\bf 1}_{p-1}) F(1, {\bf 0}_{l-1}, t_{l+1}^2, \dots, t_p^2) - F({\bf 0}_l, \bt'')^2\\
= &\ c^2 \prod_{j \in J_+} t_j^{2\alpha_j} - F({\bf 0}_l, \bt'')^2,
\end{align*}
which yields the reverse inequality via square roots, and completes  the induction step.

\item Now suppose $\mathcal{N} = 2$. The analysis in part~(1) shows one implication, except for the lack of dependence on the $t_j$ for which $k_j = 1$. Fix such an index $j_0$, and choose nonnegative scalars $t,t'$ and $t_j$ for $j \in [p] \setminus \{ j_0 \}$. We need to show that
\begin{equation}\label{Ewanted1}
0 \leq F(t_1, \dots, t_{j_0-1}, t, t_{j_0+1}, \dots, t_p) = F(t_1, \dots, t_{j_0-1}, t', t_{j_0+1}, \dots, t_p).
\end{equation}
For this, choose $x_1 < x_2$ in $X$ and let $\bx := (x_1,x_2)$. Now given an integer $n$, for all $j \in [p] \setminus \{ j_0 \}$ we set $K_j(x,x') := t_j \mathbf{1}$ for all $x,x' \in X$. For $j = j_0$, define $K_{j_0}[\bx; \bx] := \begin{pmatrix} t & t' \\ t' & t \end{pmatrix}$ and $K_{j_0} \equiv 1$ on the rest of $X \times X$. Thus $K_{j_0} \in \STN_X^{(1)}$, and so $F \circ {\bf K} \in \STN_X^{(2)}$. Computing its determinant yields:
\[
0 \leq \det (F \circ {\bf K}) [ \bx; \bx] = F(t_1, \dots, t_{j_0-1}, t, t_{j_0+1}, \dots, t_p)^2 - F(t_1, \dots, t_{j_0-1}, t', t_{j_0+1}, \dots, t_p)^2.
\]

Similarly, using $K_{j_0}[\bx; \bx] = \begin{pmatrix} t' & t \\ t & t' \end{pmatrix}$ and repeating the above argument yields the reverse inequality. This implies~\eqref{Ewanted1}.

Now one implication in~(2) is completely proved. Conversely, suppose first that $|X| = 2$ and $F$ has the specified properties in (2)(a). Let $X = \{ x < x' \}$; now given kernels $K_j \in \STN_X^{(k_j)}$, we need to verify that $(F \circ {\bf K}) [\bx; \bx]$ is $\TN_{2 \times 2}$, where $\bx = (x,x')$. Clearly, $(F \circ {\bf K} ) [\bx; \bx]$ is $\TN_{1 \times 1}$ since $F \geq 0$. Next, write $K_j[ \bx; \bx ] = \begin{pmatrix} t_j & s_j \\ s_j & t'_j \end{pmatrix}$. Restricting ourselves to those $j$ for which $k_j \geq 2$ (as $F$ does not depend on the other variables), we have $0 \leq s_j \leq \sqrt{t_j t'_j}$. Now the three given properties of $F$ yield three inequalities:
\[
0 \leq F({\bf s}) \leq F \left( (\sqrt{t_j t'_j})_j \right) \leq \sqrt{F(\bt) F(\bt')},
\]
and so $F \circ {\bf K} \in \STN_X^{(2)}$ as claimed.

If instead $|X|>2$ and $F$ is a product of power functions and Heaviside indicators as in (2)(b), then we are done using some easily verified properties of the class $\TN_{2 \times 2}$ and hence $\STN_{X \times X}^{(2)}$: it is preserved under entrywise positive (individual) powers, entrywise products, and positive rescaling.

 \item Now let $3 \leq \mathcal{N} \leq |X|$ and $F$ be nonconstant. From above, $F$ is of the form
\begin{equation}\label{Eindicator1}
F({\bf t}) = c \prod_{j \in J} t_j^{\alpha_j} \prod_{i \not\in J} {\bf 1}_{t_i > 0} 
\end{equation}
with $J\subseteq [p]$ and all $\alpha_j \geq 0$. 
We claim that Heaviside indicators do not occur, i.e.\ $J = [p]$. This is shown via the kernels and arguments around~\eqref{Esqrt2} (with $Y=X$ and $\by =\bx$), since all matrices there were symmetric. Here and below, we work with a fixed $\bx \in \inc{X}{3}$.

We next claim $k_{j_0} \geq \mathcal{N} = 3$ and $\alpha_{j_0} \geq 1$ for each $j_0 \in J$. These are already shown after~\eqref{Esqrt2} (setting $Y = X$ and $\by = \bx$), since all kernels used there were symmetric.

This proves one implication; the converse holds by \cite[Proposition~4.2 and Theorem~5.2]{FJS}.

\item Now say $4 \leq \mathcal{N} \leq \infty$. By the previous parts, $F(\bt) = c \prod_{j\in J} t_{j}^{\alpha_j}$ with $\alpha_{j}\geq 1$. First fix $\bx \in \inc{X}{4}$ and $j_0 \in J$. By \cite[Theorem~3.6(d)]{BGKP-TN}, $\alpha_{j_0} \in \{ 1 \} \sqcup [2,\infty)$ -- e.g., by considering powers of the padded-by-zero matrix $(1 + u_0^{i+j})_{i,j=1}^4 \in \STN_{4 \times 4}$ for any $u_0 \in (0,1)$; and $K_j \equiv {\bf 1}_{X \times X}$ for $j \neq j_0$.

We next prove that $J$ is a singleton (when $F$ is nonconstant). Suppose not; say e.g.\ $1,2 \in J$. Let the kernels $K_{1}$ and $K_{2}$ be the paddings-by-zero to $X \times X$ of the matrices
\[
K_1[\bx; \bx] := \begin{pmatrix}
    2 & 2 & 1 & 1 \\ 2 & 2 & 1 & 1 \\ 1 & 1 & 2 & 2 \\ 1 & 1 & 2 & 2 
\end{pmatrix}, \qquad
K_2[ \bx; \bx] := \begin{pmatrix}
    2 & 1 & 1 & 0 \\ 1 & 2 & 2 & 1 \\ 1 & 2 & 2 & 1 \\ 0 & 1 & 1 & 2 
\end{pmatrix},
\]
in \cite[Example~4.4]{FJS}; note these are in $\STN_{4\times 4}$.
Also let $K_j \equiv {\bf 1}_{X \times X}$ for $j>2$. Then $F\circ {\bf K} \in \STN_X^{(4)}$, so
\[
0\leq \det (F\circ {\bf K})[(x_1,x_2,x_3); (x_2,x_3,x_4)] = -2^{\alpha_2} (4^{\alpha_1} -1) < 0
\]
which is absurd. Hence there exists unique $j_0 \in J$ such that $F(\bt) = c t_{j_0}^{\alpha_{j_0}}$ with $\alpha_{j_0}\in \{1\}\cup [2,\infty)$. This shows one implication when $|X|=4$ -- except that $k_{j_0} \geq \mathcal{N} = 4$, which we show later. The converse assertion, with $|X|=4$, follows from \cite[Proposition~5.6]{FJS}.

\item Now suppose that $5\leq |X| \leq \infty$ (and $4 \leq \mathcal{N} \leq \infty$). By the above analysis (and padding), there exists a unique $j_0 \in J$ such that $F(\bt) = ct_{j_0}^{\alpha_{j_0}}$ with $\alpha_{j_0} \in \{1\}\cup [2,\infty)$. The powers $\alpha_{j_0} \geq 2$ are ruled out using \cite[Example 5.10]{FJS}. Hence $\alpha_{j_0} = 1$, proving one implication except that $k_{j_0} \geq \mathcal{N}$. The converse implication is obvious for $\mathcal{N} = 4$ and $5 \leq \mathcal{N} \leq \infty$.

\item The only missing part is to assume $4 \leq \mathcal{N} \leq \infty$ and show that $k_{j_0}\geq \mathcal{N}$. This is verbatim as in part~(4) of the proof of Theorem~\ref{TmainTN} (with $Y=X$ and $Y'=X'$), since the matrices used there were symmetric.\qedhere
\end{enumerate}
\end{proof}

\section{Transforms of symmetric $\TP$ kernels}

In this section we prove Theorem~\ref{TmainSTP}. This requires a final preliminary result:

\begin{lemma}\label{Lrank1approx}
Suppose there exists a $\TP_2$ kernel on $X \times Y$, for totally ordered sets $X,Y$ with $2 \leq |X|,|Y| \leq \infty$. Given any functions $\varphi : X \to (0,\infty)$ and $\psi : Y \to (0,\infty)$, the ``rank one'' $\TP^{(1)}$ kernel $K(x,y) := \varphi(x) \psi(y)$ is the limit of $\TP$ kernels on $X \times Y$ -- which are moreover symmetric if $Y=X$ and $\psi \equiv \varphi$ (i.e.\ if $K$ is symmetric).
\end{lemma}

\begin{proof}
By Lemma~\ref{Ltptosets}, $X,Y$ can be identified with subsets of $(0,\infty)$.
Now define $K^{(n)}(x,y) := e^{xy/n} K(x,y)$; these kernels (are symmetric if $Y=X$ and $\psi \equiv \varphi$, and) converge pointwise to $K$. The $K^{(n)}$ are also $\TP$ because given tuples $\bx \in \inc{X}{k}$ and $\by \in \inc{Y}{k}$, Vandermonde theory yields
\[
\det K^{(n)}[\bx; \by] = \det \left( {\rm diag}(\varphi(x_i))_{i=1}^k \cdot (e^{x_i y_j / n})_{i,j=1}^k \cdot {\rm diag}(\psi(y_j))_{j=1}^k \right) > 0. \qedhere
\]
\end{proof}

With Lemma~\ref{Lrank1approx} at hand, we now show:

\begin{proof}[Proof of Theorem~\ref{TmainSTP}]
The $\mathcal{N}=1$ case is obvious.
\begin{enumerate}
\item Before working out the $\mathcal{N}=2$ case, we first develop some results that will apply in all parts below. Thus, henceforth we assume $2 \leq \mathcal{N} \leq \infty$. First suppose $F[-]$ satisfies~\eqref{ESTPq}. Since $|X| \geq 2$, by the ``symmetric'' part of Lemma~\ref{Ltptosets} we may identify $X$ with a subset of $(0,\infty)$; and by \cite[Theorem~4.10]{BGKP-TN} we may realize every $\STP_{2 \times 2}$ matrix as a principal submatrix of a symmetric $\TP$ kernel on $X \times X$. Thus, every preserver $F[-]$ sends $(\STP_{2 \times 2})^p$ to $\STP_{2 \times 2}$. Applying Proposition~\ref{Pmidconvex} with minor modifications to the non-strict inequalities deduced, we get that $F$ is positive, jointly increasing, multiplicatively mid-convex, and hence continuous on $(0,\infty)^p$.

Next, suppose moreover that $3 \leq |X| \leq \infty$, and fix points $x_1 < x_2 < x_3$ in $X \subseteq (0,\infty)$. We show that $F$ is a mixed power function (up to rescaling). We first fix $t_2, \dots, t_p > 0$, and let $t, t' \in (0,\infty)$. Now set the totally positive kernels
\[
K_j^{(n)}(x,x') := t_j e^{x'x/n}, \qquad \forall j > 1, \ n \geq 1,
\]
and define the following rank-one matrices in $\STN_{3 \times 3}$, found in \cite[Equation (3.8)]{BGKP-TN}:
\begin{equation}\label{Etp2}
A'( u, v ) := \begin{pmatrix}
u^2 & u & u v \\
u & 1 & v \\
u v & v & v^2
\end{pmatrix} \quad \text{and} \quad %
B'(u,v) := \begin{pmatrix}
u^2 v & u v & u \\
u v & v & 1 \\
u & 1 & 1 / v
\end{pmatrix} \qquad u, v > 0.
\end{equation}

These matrices were used in understanding \textit{univariate preservers} of symmetric $\TP$ \textit{matrices}; we now use them to work out the threefold generalization: multivariate transforms of symmetric $\TP$ kernels. To do this, choose any positive function $\varphi$ that takes the values $t,1,t'$ at $x_1, x_2, x_3$ respectively, and $1$ elsewhere. Then the kernel $K_1(x,x') := \varphi(x) \varphi(x')$ equals $A(t,t')$ when evaluated at $[\bx; \bx]$ where $\bx = (x_1, x_2, x_3)$; moreover, by Lemma~\ref{Lrank1approx} $K_1$ is approximated by a sequence $K_1^{(n)} \in \STP_{X \times X}$. Hence $F \circ {\bf K}^{(n)} \in \STP_X^{(2)}$, since $\mathcal{N} \geq 2$. Evaluate the minor corresponding to $(x_1, x_2); (x_2, x_3)$ and take limits to obtain:
\[
0 \leq \det \lim_{n \to \infty} (F \circ {\bf K}^{(n)}) [ (x_1, x_2); (x_2, x_3) ] = g_1(t) g_1(t') - g_1(tt') g_1(1),
\]
where $g_1(t) := F(t, t_2, \dots, t_p)$.

A similar procedure as above, using $B'(u,v)$ in place of $A'$ shows that $g_1(tt') g_1(1) \geq g_1(t) g_1(t')$. It follows that the continuous map $g_1(\cdot)/g_1(1)$ is multiplicative on $(0,\infty)$. Thus $g_1(t) = g_1(1) t^{\alpha_1}$, with $\alpha_1 = \alpha_1(t_2, \dots, t_p)$ some function.
Similarly, each $g_{j}(t) = c_j t^{\alpha_j}$ for some functions $c_j = c_j(t_1, \dots, \widehat{t_j},\dots, t_p), \alpha_j = \alpha_j(t_1, \dots, \widehat{t_j},\dots, t_p) > 0$. Now Lemma~\ref{Lclaim} implies $F$ is a mixed power function as claimed.

\item Now suppose $\mathcal{N} = 2$. The analysis in part~(1) shows one implication, except for the lack of dependence on the $t_j$ for which $k_j = 1$. Fix such an index $j_0$, choose positive scalars $t,t'$ and $t_j>0$ for $j \in [p] \setminus \{ j_0 \}$, and proceed akin to the argument following~\eqref{Ewanted1}: choose $x_1 < x_2$ in $X \subseteq (0,\infty)$ and let $\bx := (x_1,x_2)$. Now given an integer $n$, for all $j \in [p] \setminus \{ j_0 \}$ we set $K_j^{(n)}(x,x') := t_j e^{x'x/n}$ for all $x,x' \in X$. This is a generalized Vandermonde kernel, hence in $\STP_X^{(\infty)} \subseteq \STP_X^{(k_j)}$. For $j = j_0$, define $K_{j_0} = K_{j_0}^{(n)}[\bx; \bx]$ as after~\eqref{Ewanted1}; then $K_{j_0}^{(n)} \in \STP_X^{(1)}$, and so $F \circ {\bf K}^{(n)} \in \STP_X^{(2)}$. Computing its (positive) determinant and taking limits (using that $F$ is continuous) yields:
\[
0 \leq \lim_{n \to \infty} \det (F \circ {\bf K}^{(n)}) [ \bx; \bx] = F(t_1, \dots, t_{j_0-1}, t, t_{j_0+1}, \dots, t_p)^2 - F(t_1, \dots, t_{j_0-1}, t', t_{j_0+1}, \dots, t_p)^2.
\]

As after~\eqref{Ewanted1}, using $K_{j_0}^{(n)}[\bx; \bx] = \begin{pmatrix} t' & t \\ t & t' \end{pmatrix}$ yields the reverse inequality, and hence~\eqref{Ewanted1}.

Now one implication in~(2) is fully proved, and the converse is similar to the corresponding arguments in proving Theorem~\ref{TmainSTN} -- with the obvious minor modifications, e.g.:
\[
0 < F({\bf s}) < F \left( (\sqrt{t_j t'_j})_j \right) \leq \sqrt{F(\bt) F(\bt')}.
\]

\item Next, first assume that $\mathcal{N} \geq 3$. As always, $X \subseteq (0,\infty)$. Now fix $u_0 \in (0,1)$, and appeal to the Hankel kernel $K_H(x,x') := 1 + u_0^{x+x'}$, on $(0,\infty)^2$ at first. This kernel is $\TN$ of rank~2:
\[
K_H \equiv {\bf 1}_{(0,\infty)} {\bf 1}_X^T + u_0^{(0,\infty)} (u_0^{(0,\infty)})^T,
\]
and one verifies that all $2 \times 2$ minors are nonnegative. Moreover, as explained in \cite[Proposition~7.4]{BGKP-TN}, $K_H$ is a limit of Hankel kernels $K_H^{(\epsilon)} \in \STP_{(0,\infty)}$ as $\epsilon \to 0^+$.

Continuing with the proof -- still for $3 \leq \mathcal{N} \leq \infty$ -- by part~(1) we have $F(\bt) = c \bt^{\boldsymbol \alpha}$. Fix $j_0 \in [p]$ such that $\alpha_{j_0} > 0$; we claim that $\alpha_{j_0} \not\in (0,\mathcal{N}-2) \setminus \mathbb{Z}$. For this, we use a sequence of kernel-tuples: set $K_{j_0}^{(n)} \equiv K_H^{(1/n)}$ (as above) where $n$ will grow, and also set for $j \neq j_0$:
\[
K_j(x,x') = K_j^{(n)}(x,x') := e^{x'x/n}, \qquad x,x' \in X, \ n \geq 1.
\]
Then $K_j^{(n)}|_{X \times X} \in \STP_X$ for all $j \in [p]$ and $n \geq 1$; moreover,
$K_{j_0}^{(n)} \to K_H$ and $K_j^{(n)} \to {\bf 1}$ on $X \times X$ for $j \neq j_0$,
as $n \to \infty$. Now recall (see e.g.\ the discussion following \cite[Definition~3.1]{BGKP-TN}) that $X$ has either an increasing or decreasing infinite subset. Thus, for each finite integer $1 \leq r \leq \mathcal{N}$ one can choose a tuple $\bx_r = (x_1 < \cdots < x_r)$ in $\inc{X}{r}$. Now
\begin{equation}\label{Ebookmark}
c \left( (1 + u_0^{x_j+x_k})^{\alpha_{j_0}} \right)_{j,k=1}^r = K_H^{\alpha_{j_0}}[ \bx_r; \bx_r ] = \lim_{n \to \infty} (F \circ {\bf K}^{(n)} ) [\bx; \bx] \in \TN_{r \times r},
\end{equation}
so by \cite[Corollary~3.2]{AK}, $\alpha_{j_0} \not\in (0,\mathcal{N}-2) \setminus \mathbb{Z}$. In particular, if $\mathcal{N} = \infty$ then $\alpha_{j_0} \in \mathbb{Z}_{\geq 0}$.

Finally, set $\mathcal{N} = 3$. From above we have $F(\bt) = c \prod_{j \in J} t_j^{\alpha_j}$ for some nonempty subset $J$, with all such $\alpha_j \geq 1$. It remains to show that $k_{j_0} \geq 3$ for each $j_0 \in J$. This is similar to part~(3) of the proof of Theorem~\ref{TmainTP} (setting $Y=X$), since the Jain--Karlin--Schoenberg kernel is symmetric. Namely, suppose $k_{j_0} \leq 2$; now fix $0 < \epsilon < 1$ and set
\[
K_{j_0}(x,x') = K_{j_0}^{(n)}(x,x') := (1 + xx')^{\alpha_{j_0}^{-1} \epsilon}
\quad \text{and} \quad
K_j^{(n)}(x,x') := e^{x'x/n},
\]
for all $x,x' \in X \subseteq (0,\infty)$, $n \geq 1$, and $j \neq j_0$.
Then $K_j^{(n)} \in \STP_X$ for all $j \neq j_0$, while $K_{j_0}^{(n)} \in \STP_X^{(2)}$ by \cite[Theorem~C]{AK} (or directly, being closed under positive entrywise powers). Thus $F \circ {\bf K}^{(n)} \in \STP_X^{(3)}$, and taking limits we obtain that $(x,x') \mapsto (1+xx')^\epsilon$ is $\TP^{(3)}$ when evaluated at $[\bx; \bx]$ for some $\bx \in \inc{X}{3}$. But this is false by \cite[Theorem~C]{AK} since $\epsilon \in (0,1)$.

This shows one implication; the converse holds by \cite[Proposition~4.2 and Theorem~5.2]{FJS}.

\item Now say $\mathcal{N}=4$. By the previous parts, $F(\bt) = c \prod_{j \in J} t_j^{\alpha_j}$ with $\alpha_j \in \{ 1 \} \sqcup [2,\infty)$. Thus $F$ admits a continuous extension to $[0,\infty)^p$. Moreover, the extra hypothesis yields that every $n \times n$ matrix in $\STN$ (with $n \leq |X|$) can be approximated by $\STP$ matrices via Whitney density if $X$ is finite, else its ``inflation'' to $X \times X$ can be approximated by $\STP$ kernels using Proposition~\ref{Pinflate} and Theorem~\ref{Twhitney2}. As these test kernels were the only ones used to deduce in the proof of Theorem~\ref{TmainSTN} that
(a)~$F$ is an individual power function $c t_{j_0}^{\alpha_{j_0}}$, and moreover that
(b)~if $|X|>4$ then $\alpha_{j_0} = 1$,
hence the same conclusions now follow in the $\STP$ case. Finally, we defer the proof of $k_{j_0} \geq \mathcal{N}$ to the next case below.

The converse is immediate if $|X|>4$, and follows from \cite[Proposition~5.6]{FJS} if $|X|=4$.

\item Finally, $5 \leq \mathcal{N} \leq \infty$. Again by the previous parts, $F(\bt) = c \prod_{j \in J} t_j^{\alpha_j}$ with all $\alpha_j \in \mathbb{Z}_{>0} \cup (\mathcal{N}-2,\infty)$. First if $\mathcal{N} = \infty$ then by the line after Equation~\eqref{Ebookmark}, all $\alpha_j \in \mathbb{Z}_{>0}$. Now apply Lemma~\ref{LregularTN} with $q=p$ and all $\Lambda_j$ equal to the kernel $M_\gamma$ in~\eqref{EevenPF}. Then $M_\gamma^{\sum_{j \in J} \alpha_j}$ is $\TN$. But by the discussion following~\eqref{EevenPF}, this happens only if $\sum_j \alpha_j = 1$. This completes the $\mathcal{N}=\infty$ case (the converse being obvious), except to show that each $k_j = \infty$; but this latter is shown in the proof of Theorem~\ref{TmainTP}(4) (upon setting $Y=X$), since the kernels used there were symmetric.

Else $5 \leq \mathcal{N} < \infty$. Now using the extra hypothesis, we argue as in the previous part to deduce that $F(\bt)$ is an individual power function. Moreover, the $\STN$ matrices in \cite[Example~5.10]{FJS} help rule out all powers in $(1,\infty)$. This yields the result (with the converse again obvious) except that each $k_j \geq \mathcal{N}$ -- which was also not shown in part~(4) above.

We now prove both cases together; thus, now $4 \leq \mathcal{N} < \infty$. Suppose $1 \leq k_{j_0} < \mathcal{N}$; then the arguments in the final paragraph of the proof of Theorem~\ref{TmainTP}(4) go through verbatim (with $Y=X$), since the Jain--Karlin--Schoenberg kernel is symmetric.
\qedhere
\end{enumerate}
\end{proof}

\subsection*{Acknowledgments}
This work was partially supported by SwarnaJayanti Fellowship grants SB/SJF/2019-20/14 and DST/SJF/MS/2019/3 from SERB and DST (Govt.~of India) and by a Shanti Swarup Bhatnagar Award from CSIR (Govt.\ of India). Part of this work was carried out when both authors were visiting the International Centre for Mathematical Sciences (Edinburgh), and we thank them for their excellent hospitality and working conditions.

\appendix
 \section{Discontinuities of jointly monotone functions}

 In this section, we discuss the Lebesgue measure of the set of discontinuities of a jointly monotone function on the Euclidean space $\mathbb{R}^{p}.$ We restrict our attention to the case of $(0,\infty)^{p}$ but the techniques can be used to prove the result verbatim on any subset of $\mathbb{R}^{p}.$ Let `$<$' denote the (strict) componentwise or product order on $\mathbb{R}^{p}$, that is
 $\bx < \by$ if $x_{j} < y_{j}$, for all $j\in [p].$
 
 We start by proving that a jointly non-decreasing function (recall Definition~\ref{Dstrict}) is Lebesgue (need not be Borel) measurable. The following proof from \cite{NE} is for a slightly different notion of jointly monotone maps, and it goes through verbatim in our setting. We include it here for completeness.
\begin{proposition}\label{Lebmeas}
    With the above assumptions, $f$ is Lebesgue measurable.
\end{proposition}

\begin{proof}
   We use induction on $p$.  For $p=1$ this is trivial. For the induction step, suppose $f : \mathbb{R}^{p+1} \to \mathbb{R}$ is jointly non-decreasing (as in Definition~\ref{Dstrict}), fix $a \in \mathbb{R}$, and define $g : \mathbb{R}^p \to \mathbb{R}$ by $g(x) := \inf\{t \in \mathbb{R} : f(x,t) \ge a\}$.  Then $g$ is monotonically non-decreasing.  By the induction hypothesis, $g$ is Lebesgue measurable. Hence so is the epigraph $E = \{(x,t) : t > g(x)\}$, and the graph $G = \{(x,t) : g(x) = t\}$ is of zero Lebesgue measure (Fubini's theorem).  However, $E\, \triangle\, \{f \ge a\} \subset G$ and so $\{f \ge a\}$ is Lebesgue measurable.  
\end{proof}
We next show that the set of discontinuities of a jointly monotone function has Lebesgue measure zero. A similar result was shown in \cite{BL-mntcts}, again for a somewhat different definition of joint monotonicity. The proof of that can be adapted to the present situation; we now provide an alternate and self-contained argument, to show:

 \begin{theorem}\label{Discts}
     Let $f: (0,\infty)^{p}\to \mathbb{R}$ be a jointly non-decreasing function. Then the set of discontinuities of $f$, say $D_{f},$ is of zero Lebesgue measure.
 \end{theorem}
 \begin{proof}
     Let $\lambda_{n}$ denote the $n$-dimensional Lebesgue measure on $\mathbb{R}^{n}$. We proceed by induction on $p \geq 1$: The result for $p=1$ immediately follows from the standard fact that a monotone function has at most countably many discontinuities. For $p=2$: We know that $D_f$ is measurable (being an $F_{\sigma}$ set), and via a change of variables by the linear operator $\begin{pmatrix}
         1 & -1 \\ 1 & 1
     \end{pmatrix}$ we get by Fubini's theorem,
\[
\lambda_{2}(D_f) =\int_{-\infty}^{+\infty}\int_{-\infty}^{+\infty} 1_{D_f}(x,y) \ \mathrm{d}x\mathrm{d}y \\
=2\int_{-\infty}^{+\infty}\left[\int_{-\infty}^{+\infty} 1_{D_f}(s-t,s+t) \ \mathrm{d}s\right]\mathrm{d}t .\\
\]

Fix $t$ and let $\varphi_t : \mathbb{R} \to \mathbb{R}$ be the function mapping $s$ to $f(s-t,s+t)$. We claim that if $\varphi_t$ is continuous at $s$, then $f$ is continuous at $(s-t,s+t)$. Indeed, since $f$ is strictly increasing we have
\[
f(s-t-\delta,s+t-\delta) \leq f(x,y) \leq f(s-t+\delta,s+t+\delta)
\]
for all $(x,y)$ in the open square $(s-t-\delta,s-t+\delta)\times(s+t-\delta,s+t+\delta)$. In the above notation,
\[
\varphi_t(s - \delta) \leq f(x,y) \leq \varphi_t( s + \delta), 
\]
so taking the limit as $\delta \to 0^{+}$ proves the claim.

In other words, a discontinuity of $f$ is a discontinuity of $\varphi_t$. Therefore, for each $t$, we have $D_{f}\cap l \subseteq D_{\varphi_{t}}$, where $l= \{(s-t,s+t)^{T} \in \mathbb{R}^{2} : s \in \mathbb{R} \}$ is a line. But $\varphi_t$ is non-decreasing, so by the base case $\lambda_{1}(D_{\varphi_t}) = 0$. Therefore, we have the following inequality,
\[
\lambda_{2}(D_{f}) = 2\int_{-\infty}^{+\infty}\left[\int_{-\infty}^{+\infty} 1_{D_f}(s-t,s+t) \ \mathrm{d}s\right]\mathrm{d}t \leq 2 \int_{-\infty}^{+\infty}\left[\int_{-\infty}^{+\infty} 1_{D_{\varphi_t}}(s) \ \mathrm{d}s\right]\mathrm{d}t.
\]
Since the inner integral on the RHS is $0$ for all $t$, Fubini's theorem implies that $\lambda_{2}(D_{f})=0$.

For the induction step, a similar argument applies by taking the matrix $A = (a_{ij})_{i,j=1}^{p}$ as a change of variable matrix, where $a_{ij} = -1$ if $i+j=p+1, 1\leq i<p$ and $1$ otherwise. This yields
\[
\lambda_{n}(D_f) =\int_{\mathbb{R}^{p}} 1_{D_f}(x_{1},\ldots, x_{p}) \ \mathrm{d}x_{1}\mathrm{d}x_{2}\cdots \mathrm{d}x_{p} \\
= 2^{p-1} \int_{\mathbb{R}^{p}} 1_{D_f}(A{\bf s}) \ \mathrm{d}s_{1}\mathrm{d}s_{2}\cdots \mathrm{d}s_{p},
\]
where $A {\bf s} = {\bf x}$.
Again by fixing $(s_{2},\ldots, s_{p}) \in \mathbb{R}^{p-1},$ we define $\varphi_{(s_{2},\ldots, s_{p})} = \varphi: \mathbb{R} \to \mathbb{R}$ which sends $s_{1}$ to $f(A{\bf s})$. Now using the fact that $f$ is monotonically increasing, we get that
\[
\varphi(s_{1} - \delta) \leq f(\bx) \leq \varphi(s_{1} + \delta)
\]
for all $\bx$ in the open cube $(\langle r_{1}, {\bf s} \rangle - \delta, \langle r_{1}, {\bf s} \rangle - \delta) \times(\langle r_{2}, {\bf s} \rangle -\delta, \langle r_{2}, {\bf s} \rangle + \delta) \times \cdots \times(\langle r_{p}, {\bf s} \rangle - \delta, \langle r_{p}, {\bf s} \rangle + \delta)$, where $\langle r_{j}, {\bf s} \rangle$ denotes the Euclidean inner product of the $j^{th}$ row vector of $A$ with ${\bf s}$.
Therefore $D_{f} \cap l \subseteq D_{\varphi}$, where $l= \{A{\bf s} \in \mathbb{R}^{p} : s_{1} \in \mathbb{R} \}$ is a line in $\mathbb{R}^{p}$. By Fubini's theorem,
\[
\lambda_{p}(D_{f}) \leq 2^{p-1}  \int_{\mathbb{R}^{p-1}} \bigg[\int_{\mathbb{R}}1_{D_\varphi}(s_{1}) \ \mathrm{d}s_{1}\bigg] \mathrm{d}s_{2}\cdots \mathrm{d}s_{p},
\]
where the factor $2^{p-1} = \det(A)$. By the induction hypothesis, the inner integral vanishes for all choices of $(s_{2},\ldots, s_{p})$, and hence $\lambda_{p}(D_{f}) = 0$.
\end{proof}

\end{document}